\documentclass[12pt]{amsart}

\usepackage{amssymb,amsthm,amsmath,amstext,amsxtra}
\usepackage{mathrsfs}
\usepackage{mathtools, bbm}
\usepackage{enumitem}
\usepackage{comment}
\usepackage{colonequals}
\usepackage{float} 
\usepackage{tikz-cd}

\usepackage{thmtools}
\usepackage{thm-restate}

\usepackage[margin=1in]{geometry}

\usepackage{color}

\definecolor{darkgreen}{rgb}{0,0.5,0}
\usepackage[pagebackref=true,
colorlinks, citecolor=darkgreen,
]{hyperref}

\newcommand{\OKS}{R}
\newcommand{\OKSp}{R'}

\newcommand{\bbQ}{\mathbb{Q}}

\newcommand{\bbZ}{\mathbb{Z}}
\newcommand{\bbG}{\mathbb{G}}

\newcommand{\bbR}{\mathbb{R}}
\newcommand{\bbC}{\mathbb{C}}
\newcommand{\bbP}{\mathbb{P}}
\newcommand{\bbA}{\mathbb{A}}

\DeclareMathOperator{\cJ}{\mathcal{J}}

\DeclareMathOperator{\cO}{\mathcal{O}}

\DeclareMathOperator{\fkm}{\mathfrak{m}}

\DeclareMathOperator{\fkp}{\mathfrak{p}}
\DeclareMathOperator{\fkP}{\mathfrak{P}}

\DeclareMathOperator{\Res}{Res}
\DeclareMathOperator{\Spec}{Spec}
\DeclareMathOperator{\rank}{rank}

\DeclareMathOperator{\Span}{\mathrm{Span}}

\DeclareMathOperator{\cA}{\mathcal{A}}

\DeclareMathOperator{\cC}{\mathcal{C}}

\DeclareMathOperator{\supp}{\mathrm{supp}}

\DeclareMathOperator{\Gal}{\mbox{Gal}}

\newtheorem{theorem}{Theorem}[section]
\newtheorem{lemma}[theorem]{Lemma}
\newtheorem{corollary}[theorem]{Corollary}
\newtheorem{proposition}[theorem]{Proposition}

\theoremstyle{definition}
\newtheorem{definition}[theorem]{Definition}
\newtheorem{conjecture}[theorem]{Conjecture}
\newtheorem{question}[theorem]{Question}

\theoremstyle{remark}
\newtheorem{algorithm}[theorem]{Algorithm}
\newtheorem{remark}[theorem]{Remark}
\newtheorem{example}[theorem]{Example}

\title{Restriction of Scalars Chabauty and the $S$-Unit Equation}

\author{Nicholas Triantafillou}
\email{nicholas.triantafillou@uga.edu}
\address{Department of Mathematics, University of Georgia\\
		Athens, GA 30602, USA}
\date{June 28, 2021}
\thanks{I thank the National Science Foundation Graduate Research Fellowship and Research Training Group in Algebra, Algebraic Geometry, and Number Theory at the University of Georgia [grant numbers 1122374, DMS-1344994]; and the Simons Foundation [grant number 550033] for funding.}

\begin{document}

\maketitle

\begin{abstract}
Given a smooth, proper, geometrically integral curve $X$ of genus $g$ with Jacobian $J$ over a number field $K$, Chabauty's method is a $p$-adic technique to bound $\# X(K)$ when $\rank J(K) < g$. We study limitations of a variant called `Restriction of Scalars Chabauty' (RoS Chabauty). RoS Chabauty typically bounds $\# X(K)$ when $\rank J(K) \leq [K:\mathbb{Q}] (g - 1)$, but fails in the presence of a \emph{subgroup obstruction}, a high-rank subgroup scheme of $\Res_{K/\mathbb{Q}} J$ which intersects the image of $\Res_{K/\mathbb{Q}} X$ in higher-than-expected dimension. We define \emph{BCP obstructions}, which are certain subgroup obstructions arising from the geometry of $X$. BCP obstructions explain all \emph{known} examples where RoS Chabauty fails to bound $\# X(K)$. We also extend RoS Chabauty to compute $S$-integral points on affine curves. 

Suppose $K$ does not contain a CM-subfield. We present a $p$-adic algorithm which conjecturally computes solutions to the $S$-unit equation $x+y = 1$ for $x,y \in \mathcal{O}_{K,S}^{\times}$ by using RoS Chabauty to compute $S$-integral points on certain genus $0$ affine curves. As evidence the algorithm succeeds, we prove that all but one of these curves have no subgroup obstructions and that the remaining curve has no BCP obstructions to RoS Chabauty. In contrast, under a generalized Leopoldt conjecture, we prove that analogous methods using classical Chabauty cannot bound solutions to the $S$-unit equation when $[K:\mathbb Q] \geq 3$ and $K$ is not totally real.
\end{abstract}

\section{Introduction}

\subsection{Background}\label{sec:chabauty-background}

Let $X$ be a smooth, proper, geometrically irreducible curve of genus $\geq 2$ over a number field $K$.\footnote{We will often take $X$ to be a suitable $S$-integral model of a smooth, geometrically irreducible affine curve instead of a smooth proper curve over a number field $K$ of genus $\geq 2$.} By Faltings' Theorem, $X(K)$ is finite \cite{faltings-83}. The known proofs do not give a strategy to compute $X(K)$. Let $\fkp$ be a prime of $K$ and let $j: X \to J$ be a finite map from $X$ to its Jacobian. Under the additional assumption 
\begin{align}\label{eqn:chabauty-inequality-1}
\rank J(K) < \dim J = \text{genus}(X)\,,
\end{align}
Chabauty's method is a $\fkp$-adic technique to produce an explicit finite subset $X(K_{\fkp})_{1} \subset X(K_{\fkp})$ such that $X(K) \subset X(K_{\fkp})_{1}$. When $K_{\fkp} \cong \bbQ_{p}$, the set $X(K_{\fkp})_{1}$ is the intersection of $X(K_{\fkp})$ with the $p$-adic closure of $J(K)$ inside of $J(K_{\fkp})$. In general $X(K_{\fkp})_{1}$ contains this intersection. When \eqref{eqn:chabauty-inequality-1} holds, the set $X(K_{\fkp})_{1}$ is finite and one can frequently compute $X(K)$ exactly from $X(K_{\fkp})_{1}$ using techniques like the Mordell-Weil sieve. Historically, Chabauty's method was inspired by earlier work of Skolem, who computed $S$-integral points on affine genus $0$ curves by embedding them in well-chosen tori, which we recognize today as the generalized Jacobians of these curves. In fact, it is possible to synthesize Chabauty's insight with Skolem's original ideal. Given a finite set $S$ of primes of $K$, Chabauty's method can be modified to compute the $S$-integral points on a (suitable) model of an affine curve $X$ under a similar rank versus dimension hypothesis for the generalized Jacobian of $X$. To extend Chabauty's method to this case, we develop technical tools to deal with integral points on locally of finite type N\'{e}ron models of generalized Jacobians. We explain this in more detail in Sections~\ref{sec:prelim-defs} and \ref{sec:classical-chabauty}.

Several strategies have been proposed to augment Chabauty's method to compute $X(K)$ when $\dim J \leq \rank J(K) $. In this article, we focus on Restriction of Scalars Chabauty (RoS Chabauty) and descent. We follow two main directions. Our first aim is to develop a framework for understanding obstructions that prevent RoS Chabauty from proving finiteness of rational/$S$-integral points on curves in particular examples.
Our second aim is to apply RoS Chabauty and descent towards the study (by elementary $p$-adic methods) of $S$-integral points on $\bbP^1_{\mathcal O_{K,S}} \smallsetminus \{0,1,\infty\}$, or equivalently the set
$
E_{K,S} \colonequals \{(x,y) \in (\mathcal O_{K,S}^{\times})^2: x + y = 1\}
$
of solutions to the $S$-unit equation. The application of RoS Chabauty to $S$-unit equations has already inspired new results on unit equations in \cite{triantafillou-20}, which proves that if $S = \emptyset$, the degree $[K:\bbQ]$ is not a multiple of $3$, and $3$ splits completely in $K$, then $E_{K,S} = \emptyset$. In other words, the unit equation has no solutions in $K$. 

\subsection{Development of RoS Chabauty}

One strategy for extending Chabauty's method beyond the bound \eqref{eqn:chabauty-inequality-1} is to replace $X$ and $J$ with the restrictions of scalars $\Res_{K/\bbQ} X$ and $\Res_{K/\bbQ} J$. This approach, which we call RoS Chabauty, first appeared in print in \cite{siksek-13}, which attributes the idea to \cite{wetherell-00}. Let $X(K \otimes \bbQ_p)_{1}$ be the intersection of $(\Res_{K/\bbQ} X)(\bbQ_{p})$ and the $p$-adic closure of $(\Res_{K/\bbQ} J)(\bbQ)$ in $(\Res_{K/\bbQ} J)(\bbQ_p)$. An expected dimension calculation suggests that $X(K \otimes \bbQ_p)_{1}$ is finite when 
\begin{align}\label{eqn:RoS-Chabauty-inequality-1}
\rank J(K) \leq [K:\bbQ] \cdot (\dim J - 1)\,.
\end{align} 
Since $X(K) \subset X(K \otimes \bbQ_p)_{1}$, if $X(K \otimes \bbQ_p)_{1}$ is finite, then $X(K)$ is finite as well.
However, if there is a subgroup scheme $T \subset \Res_{K/\bbQ} J$ such that $\rank T(\bbQ)$ is `large' and $T$ has a translate which intersects $\Res_{K/\bbQ} X$ in positive dimension, the set $X(K \otimes \bbQ_p)_{1}$ of points cut out by RoS Chabauty is infinite even when \eqref{eqn:RoS-Chabauty-inequality-1} holds. Recent works of Dogra and Hast \cite{dogra-19, hast-19} make progress towards proving that this is the only reason why the set $X(K \otimes \bbQ_p)_{1}$ can be infinite.\footnote{In fact, \cite{dogra-19} and \cite{hast-19} work in a much broader context combining ideas of RoS Chabauty with Kim's nonabelian Chabauty program.} So, to understand RoS Chabauty it is critical to understand the arithmetic of subgroup schemes of  $\Res_{K/\bbQ} J$ and the geometry of their intersection with $\Res_{K/\bbQ} X$.

\subsubsection{\textbf{Contributions of this article}}

Let $S_{0}$ be a finite set of finite places of $\bbQ$. Let $S$ be the set of places of $K$ lying above $S_{0}$. Let $R = \cO_{K,S}$ be the ring of $S$-integers and let $R_{0}$ be the ring of $S_{0}$-integers in $\bbQ$.

In Section~\ref{sec:setup}, we explain how to use RoS Chabauty to bound  $R$-points on suitable\footnote{Here, `suitable' means `equipped with a map to the connected component of the identity of the locally of finite type N\'eron model of the generalized Jacobian of the generic fiber.}$R$-models of affine curves and develop a framework for understanding obstructions that can cause the set $X(\cO_K \otimes \bbZ_p)_{S,1}$ of points cut out by RoS Chabauty to be infinite.

Let $X/R$ be a regular model of a not-necessarily proper curve over $K$, and let $J$ be its generalized Jacobian. We have already see that $X(\cO_{K} \otimes \bbZ_{p})_{S,1}$ will be infinite if there is a translate $x\cdot T$ of a subgroup of $T \subset \Res_{R/R_{0}} J$ such that $T(R_0)$ has large rank relative to $\dim T$ and $\dim((x\cdot T) \cap j(\Res_{R/R_{0}} X)$ is large. (See Definition~\ref{def:subgroup-obstruction} for a precise definition.)

The existence of more general obstructions forcing $\#X(\cO_{K} \otimes \bbZ_{p})_{S,1} = \infty$ has not been disproved outside of special cases. However, by \cite[Corollary~2.2]{dogra-19}, such obstructions would require the existence of (translates of) subgroup schemes $T_{p}$ of $(\Res_{R/R_{0}} J)_{\bbZ_{p}}$ which do not come from a subgroup scheme over $\bbZ$ and which intersect \emph{both} $j(\Res_{R/R_{0}} X)(\bbZ_{p})$ and $\Res_{R/R_{0}} J(\bbQ_{p})$ in relatively large dimension compared to $T_{p}$. (Roughly, \cite[Corollary~2.2]{dogra-19} applies Ax-Schanuel to show that there are no `unlikely intersections' so that the $T_{p}$ must be subgroup schemes rather than analytic subgroups.) Our point of view is that even if some such obstructions exist by coincidence for some prime $p$, it seems very unlikely that such obstructions would exist for most/all primes $p$ simultaneously in the absence of a subgroup obstruction.

In fact, in all examples in the literature with $\# X(\cO_{K} \otimes \bbZ_{p})_{S,1} = \infty$, there is a subgroup obstruction where the subgroup scheme $T \subset \Res_{R/R_{0}} J$ has a very special structure which is explained by the geometry of the curve $X$. To develop our framework for understanding obstructions to RoS Chabauty, we define the class of BCP\footnote{BC stands for base change and P stands for Prym variety}subgroup schemes of $\Res_{K/\bbQ} J$, where $J$ is the generalized Jacobian of $X$, which attempt to capture this structure. BCP subgroups are built out of restrictions of scalars of generalized Jacobians of curves which admit a non-constant map from $X$ after a suitable base change and generalized Prym varieties of morphisms between such curves. By construction, nontrivial BCP subgroups always intersect $\Res_{K/\bbQ} X$ in larger-than-expected dimension. This makes them candidates to obstruct RoS Chabauty. A BCP subgroup $T$ is a BCP obstruction if $T$ is also a subgroup obstruction. We give the precise definitions in Section~\ref{sec:main-defs}. Although we present these definitions in the setting where $J$ is the generalized Jacobian of an $R$-model of an affine genus $0$ curve, they generalize immediately to the higher genus and/or proper case.

The class of BCP obstructions is broad enough to include all \emph{known} examples where the set $X(\cO_{K} \otimes \bbZ_{p})_{S,1}$ cut out by RoS Chabauty is not finite. This includes examples where \eqref{eqn:RoS-Chabauty-inequality-1} does not hold, examples in Section~2 of \cite{siksek-13} where $X$ is the \emph{base change} of a curve for which \eqref{eqn:RoS-Chabauty-inequality-1} does not hold, and the more involved example in Section~2.2 of \cite{dogra-19}. BCP obstructions are also amenable to study. For any $K$ not containing a CM-subfield, Theorem~\ref{thm:main-theorem} can be used to produce infinitely many curves $X/R$ such that $\Res_{R/R_{0}} J$ has many large rank subgroups, but such that $X$ has no BCP obstruction to RoS Chabauty. Theorem~\ref{thm:main-theorem-2} goes further and shows that twists of closely related curves have no subgroup obstructions to RoS Chabauty. The difference between the two cases is that in the setup of Theorem~\ref{thm:main-theorem-2}, the Jacobian splits over a nonabelian extension of $K$, while in the setup of Theorem~\ref{thm:main-theorem}, the generalized Jacobian splits over an abelian extension of $K$. The nonabelian Galois group places strong constraints on the interaction between the arithmetic of subtori of the generalized Jacobian. We note that this seems to use nonabelian structure in a slightly different way than nonabelian Chabauty does. In Section~\ref{sec:descent-summary}, we discuss an application of these twist families to computing solutions to the $S$-unit equation in $K$.

Our definition of BCP obstructions raises the following natural question. 
\begin{question} \label{que:is-bcp-all-there-is}
	Is there a curve $X/K$ such that RoS Chabauty cannot prove $X(K)$ is finite, but there is no BCP obstruction to RoS Chabauty for $X$?
\end{question}
See also Remark~2.2. of \cite{dogra-19} for a similar question.
An affirmative answer to Question~\ref{que:is-bcp-all-there-is} would require a larger-than-expected intersection between $\Res_{K/\bbQ} X$ and a translate of a subgroup scheme $T$ of $(\Res_{K/\bbQ} J)_{\bbQ_{p}}$, which also intersects the closure of the integral points in large dimension. 

A negative answer would give a finite criterion for the success of RoS Chabauty involving only the arithmetic and geometry of $X$. We note that even if there are subgroup obstructions which are not BCP obstructions, Question~\ref{que:is-bcp-all-there-is} may have a negative answer. For instance, if $\Res_{R_{0}/R} J$ has a subgroup $T'$ such that $\overline{T(\bbZ)}$ has finite index in $T'(\bbZ_{p})$ then the product of $T'$ with a BCP obstruction $T$ will be a subgroup obstruction. The resulting subgroup $T\cdot T'$ may be a BCP obstuction, but even if it is not, it is \emph{explained} by the presence of the BCP obstruction $T$. Since we focus on proving the non-existence of BCP obstructions and subgroup obstructions, we have elected not to attempt to catalog ways in which subgroup obstructions could be constructed by slight modifications of BCP obstructions. 

\subsection{Descent, Chabauty variants, and the $S$-unit equation} \label{sec:descent-summary}

Descent, named for Fermat's infinite descent, is another strategy which can be used to push Chabauty's method beyond the bound \eqref{eqn:chabauty-inequality-1}. Descent replaces the curve $X/K$ with a collection of covers $f_{i} : X_{i} \to X$, each defined over $K$, with the property that $X(K) = \bigcup_{i} f_{i}(X_{i}(K))$. Since $\text{genus}(X)~\geq~2$, the covers $X_{i}$ will have higher genus than $X$ by Riemann-Hurwitz.\footnote{In the analogous setup to compute integral points on affine curves, either the genus or number of punctures will grow (or both).} One hopes that \eqref{eqn:chabauty-inequality-1} is satisfied for all of the $X_i$ so that Chabauty's method can be used to compute each $X_{i}(K)$. Covering collections can be described explicitly. For instance, taking $[n]$ to be the multiplication-by-$n$ map on $J$, a covering collection can be constructed as translates by representatives for $J(K)/nJ(K)$ of the pullback $[n]^* X$. One can also use descent to study $S$-integral points on hyperbolic affine curves.

When $K = \bbQ$, \cite{poonen-19} uses the versions of descent and Chabauty's method for $S$-integral points on genus $0$ affine curves to give a new elementary $p$-adic proof that the set $E_{\bbQ,S}$ of solutions to the $S$-unit equation is finite. The argument produces a collection of punctured genus $0$ covers of $\bbP^1_{R_0}\smallsetminus\{0,1,\infty\}$ such that the ranks of the $S$-integral points on their generalized Jacobians can be bounded explicitly. 

\subsubsection{\textbf{Contributions of this article}}

Section~\ref{sec:classical-chab+descent}  studies limitations to the strategy of \cite{poonen-19} for proving finiteness of solutions to the $S$-unit equation when $[K:\bbQ] \geq 3$ and $K$ is not totally real.  Sections~\ref{sec:RoS-chab+descent} and \ref{sec:RoS-chab+descent-2} provide evidence that RoS Chabauty may allow us to overcome these limitations. 

Suppose the following generalization of Leopoldt's Conjecture holds.\footnote{The usual Leopoldt conjecture (that the $p$-adic regulator of $L$ does not vanish) is the $K = \bbQ$ case of Conjecture \ref{conj:general-leopoldt}, since $\rank T(\bbZ) \leq \dim \bbZ$ and every torus $T/\bbQ$ is a subtorus of some $\Res_{K/\bbQ} \bbG_{m,K}$.}

\begin{conjecture}[Generalized Leopoldt over $K$] \label{conj:general-leopoldt}
	Let $T$ be an irreducible torus over $K$ and let $\fkp$ be a prime of $K$ which is large enough that the ${\fkp}$-adic logarithm map 
	\[
	\log: T \to K_{\fkp}^{\dim T}
	\]
	is well-defined. Then,
	\begin{align}\label{eqn:generalized-Leopoldt}
	\dim_{K_{\fkp}} \Span_{K_{\fkp}} \log T(\cO_{K}) = \min( \rank T(\cO_{K}^{\times}), \dim T)\,.
	\end{align}
\end{conjecture}

Let $X/\cO_{K}$ be a regular model of a punctured genus $0$ curve and let $J$ be its generalized Jacobian. When $K_{\fkp} = \bbQ_{p}$, Conjecture~\ref{conj:general-leopoldt} says that the closure $\overline{J(\cO_{K})}$ in $J(\cO_{K_{\fkp}})$ under the $\fkp$-adic topology is as large as possible given the rank. More generally, Conjecture~\ref{conj:general-leopoldt} implies that if $X(\cO_{K_{\fkp}})\neq \emptyset$ then the set $X(\cO_{K_{\fkp}})_{\emptyset,1}$ of $\fkp$-adic points produced by Chabauty's method
is finite if and only if there is some subtorus $T \subset J$ such that $\rank T(\cO_{K}) < \dim T$. 

The main result of Section~\ref{sec:classical-chab+descent} is Theorem~\ref{thm:classical-Chabauty}, shows that  if $[K:\bbQ] \geq 3$ and $K$ is not totally real, then no such torus exists, and as a consequence $X(\cO_{K_{\fkp}})_{\emptyset,1} = X(\cO_{K_{\fkp}})$.  In particular, under the generalized Leopoldt conjecture, the strategy of \cite{poonen-19} to compute solutions to the $S$-unit equation by computing $S$-integral points on genus $0$ covers of $\bbP^1\smallsetminus \{0,1,\infty\}$ will not produce a finite set of points. The main ingredient in the proof of Theorem~\ref{thm:classical-Chabauty} is Lemma~\ref{lem:rank-anisotropic}, which gives a lower bound on the rank of the $O_{K}$-points of a torus.  

Section~\ref{sec:RoS-chab+descent} is devoted to proving our main results, Theorem~\ref{thm:main-theorem} and Theorem~\ref{thm:main-theorem-2}.

\begin{restatable*}{theorem}{noBCPobstruction}
	\label{thm:main-theorem} 
	Let $q$ be a prime number. Let $K$ be a number field which does not contain a CM subfield. Let $S_{0}$ be a finite set of finite places of $\bbQ$ and let $S$ be the set of places of $K$ lying above $S_{0}$. Let $R = \mathcal O_{K,S}$ be the ring of $S$-integers. For  $\alpha \in \OKS^{\times}$ which is a $q$th power in $\OKS$, set 
	\[
	\fkm_{\alpha, q} \colonequals \{x\in \overline{K}: x^q - 1 = 0, x\neq 1\}\,,
	\]
	viewed as a divisor on $\bbP^1_{K}$. Let $\Gamma_{\alpha, q}$ be the closure of $\supp(\fkm_{\alpha,q})$ in $\bbP^1_{R}$. Set $X_{\alpha, q} \colonequals \bbP^1_{R} \smallsetminus \Gamma_{\alpha,q}$.
	
	For $q$ sufficiently large (depending on $K$ and $S$, but not on $\alpha$), there are no BCP obstructions to RoS Chabauty for $X_{\alpha,q}$. 
\end{restatable*}

\begin{restatable*}{theorem}{nosubgroupobstruction}
	\label{thm:main-theorem-2} 
	Let $q$ be a prime number. Let $K$ be a number field. Let $S_{0}$ be a finite set of finite places of $\bbQ$ and let $S$ be the set of places of $K$ lying above $S_{0}$. Let $R = \mathcal O_{K,S}$ be the ring of $S$-integers. For $\alpha \in \OKS^{\times}$ which is \emph{not} a $q$th power in $\OKS$, set 
	\[
	\fkm_{\alpha, q} \colonequals 
	\{x\in \overline{K}: x^q - \alpha = 0\}\,,
	\]
	viewed as a divisor on $\bbP^1_{K}$. Let $\Gamma_{\alpha, q}$ be the closure of $\supp(\fkm_{\alpha,q})$ in $\bbP^1_{R}$. Set $X_{\alpha, q} \colonequals \bbP^1_{R} \smallsetminus \Gamma_{\alpha,q}$.
	
	For $q$ sufficiently large (depending on $K$ and $S$, but not on $\alpha$), there are no subgroup obstructions to RoS Chabauty for $X_{\alpha,q}$. 
\end{restatable*}

For each sufficiently large $q$, Theorems~\ref{thm:main-theorem}~and~\ref{thm:main-theorem-2} construct an explicit finite set of affine genus $0$ curves $\{X_{\alpha,q}\}_{\alpha}$ such that $X_{1,q}$ has no BCP obstructions to RoS Chabauty and $X_{\alpha,q}$ for $\alpha \neq 1$ has no subgroup obstructions to RoS Chabauty. The set $\{X_{\alpha,q}\}_{\alpha}$ becomes a covering collection $\{X_{\alpha,q}'\}_{\alpha}$ for $\bbP^1_R \smallsetminus \{0,1,\infty\}$ after removing `sections' at $0$ and $\infty$, as well as the section at $1$ if $\alpha=1$. Removing sections can only increase the number of $R$-points. Hence, Theorems~\ref{thm:main-theorem}~and~\ref{thm:main-theorem-2} suggest an elementary $p$-adic procedure (generalizing the strategy in \cite{poonen-19}) which is likely to compute the set $E_{K,S}$ of solutions to the $S$-unit equation in $K$.

\begin{algorithm}[Conjectured algorithm to compute $E_{K,S}$] \label{algorithm}
	\ \\
	Given: A number field $K$ not containing a CM-subfield and a set $S$ of places of $K$. \\
	Compute: The set $E_{K,S}$ of solutions to the $S$-unit equation $x + y = 1$ in $\cO_{K,S}^{\times}$. \\
	Algorithm: Using the notation and bounds from Theorems~\ref{thm:main-theorem} and \ref{thm:main-theorem-2}:
\begin{enumerate}
	\item Choose a suitable $q$ and compute the curves $X_{\alpha, q}$.
	\item Using RoS Chabauty and the Mordell-Weil sieve, compute each $X_{\alpha, q}(R)$.
	\item Compute $X_{\alpha, q}'(R)$ as a subset of $X_{\alpha, q}(R)$ by throwing out any sections which intersect the removed sections.
	\item Compute $(\bbP^1\smallsetminus\{0,1,\infty\})(R)$ as the union of the images of the $X_{\alpha, q}'(R)$.
\end{enumerate}
\end{algorithm}

The assumption in Theorem~\ref{thm:main-theorem} that $K$ does not contain a CM-subfield is essential. In Remark~\ref{rem:CM-field-RoS-Obstruction}, we show that if $K$ has a CM-subfield, the curve $X_{1,q}$ has a BCP obstruction for any choice of $q$. We expect that it is possible to overcome this obstacle to computing solutions to the $S$-unit equation via descent and RoS Chabauty by \emph{iterating} the descent. In other words, we believe it should be possible to construct a genus $0$ covering collection for $X_{1,q}$ with no subgroup obstructions even if $K$ contains a CM-subfield. We do not pursue that line of study in this article.

The main ingredient in the proof of Theorem~\ref{thm:main-theorem} is a bound on the ranks of BCP subtori coming from a careful analysis of the possible positions in the complex plane of the images of the $q$th roots of a fixed $\alpha \in R$ under an automorphism of $\bbP^1_{K}$. Controlling the number of these images which lie on the real axis is particularly important. The restriction that $K$ does not contain a CM subfield arises because fields containing a CM-subfield are characterized by the existence of an automorphism of $\bbP^1_K$ which maps the complex unit circle to the real axis in every embedding $\iota: K \hookrightarrow \bbC$.

The main ingredient in the proof of Theorem~\ref{thm:main-theorem-2} is an exact computation of the ranks of the $\mathbb Z$-points of subtori of $\Res_{\cO_{K[\sqrt[q]]{\alpha}}/\mathbb Z} \mathbb G_{m} \smallsetminus \Res_{\cO_{K}/\mathbb Z} \mathbb G_{m}$. This formula comes from a careful study of the action of Galois on character groups of subtori together with the classification of irreducible representations of $\bbZ/q\bbZ \rtimes \bbZ/2\bbZ$. In particular, we note that Theorem~\ref{thm:main-theorem-2} does not require any assumption on $K$.

\subsection{Additional context on the $S$-unit equation}

The set $E_{K,S}$ of solutions to the $S$-unit equation is finite, due to  \cite{siegel-21} in the case $K = \bbQ$ and \cite{lang-60} for general $K$, although this case is often attributed to Mahler from the 1930s. The problem of computing/bounding solutions to the $S$-unit equation has remained of substantial interest because of a wide range of applications in number theory and related fields. In arithmetic geometry, computing the set $E_{K,S}$ is roughly equivalent to computing the set of elliptic curves with good reduction outside a fixed set of primes. In dynamics, studying $E_{K,S}$ is closely related to studying periodic points of odd order in arithmetic dynamical systems. $S$-unit equations are also important in many other fields. See \cite{evertse1988s} and \cite{evertse2015unit} for a more thorough discussion of these applications.

The best known general upper bound on $\# E_{K,S}$ is $3\cdot 7^{2\#S+3 r_{1}+4r_{2}}$ where $r_{1}$ and $r_{2}$ are the number of real and complex embeddings respectively of $K$ \cite{evertse-84}. The bound in \cite{evertse-84} is not effective. (It does not give any upper bound on the height of elements of $E_{K,S}$ or any other algorithm to compute $E_{K,S}$.) The true upper bound on $\# E_{K,S}$ is typically expected to be subexponential in $\# S, r_{1},$ and $r_{2}$. 

More recent work, based on Baker's theory of linear forms in logarithms, gives effective bounds which can be used to compute $E_{K,S}$. The current state of the art (from \cite{gyory-19}) gives a bound in terms of the number of real and complex embeddings, class number, and regulator of $K$, as well as the third-largest norm of a prime ideal in $S$. These bounds are still exponential in the parameters. An algorithm based on these ideas (and many computational tools, including the LLL algorithm for lattice basis reduction) is now implemented in the Sage computer algebra system \cite{akmrvw19}. This implementation can be used to compute quickly compute $E_{K,S}$, at least when $[K:\bbQ]$ and $\# S$ are not too large. 

The author has also studied $E_{K,\emptyset}$ using a different approach based on RoS Chabauty in \cite{triantafillou-20}. In that work, the strategy was to intersect $\Res_{\cO_{K}/\bbZ} \bbP^1\smallsetminus \{0,1,\infty\}$ with the $p$-adic closure of well-chosen subgroups of $(\Res_{\cO_{K}/\bbZ} \bbG_{m}^2)(\cO_{K})$. When $3$ does not divide $[K:\bbQ]$ and $3$ splits completely in $K$, this intersection is empty. Taking an infinite union over sets cut out by subgroups, one can conclude $(\bbP^1\smallsetminus \{0,1,\infty\})(\cO_{K}) = \emptyset$ even though the set of $p$-adic points cut out by RoS Chabauty is infinite. In contrast, the present work applies RoS Chabauty in combination with descent to study $E_{K,S}$ and is not related to \cite{triantafillou-20} beyond the fact that both use techniques building on RoS Chabauty, study $E_{K,S}$ and have their roots in the author's PhD thesis. 

Recently, several authors have given new proofs of finiteness of the $S$-unit equation as a proof-of-concept of strategies to prove an effective version of Faltings' theorem. For instance, the first proof of finiteness by Kim's nonabelian Chabauty showed that $E_{K,S}$ is finite when $K = \bbQ$ \cite{kim-05}. The method has since been used to prove that $X(\bbQ)$ is finite whenever $X$ of genus $\geq 2$ is a solvable cover of $\bbP^1$ \cite{ellenberg-hast-17}. These results have since been extended to general number fields in \cite{dogra-19} and simultaneously, under Klinger's Ax-Schanuel conjecture, in \cite{hast-19}. Effective versions of the method have been used to compute the set of rational points on the `cursed' split Cartan modular curve of level $13$ \cite{balakrishnan-dogra-muller-tuitman-vonk-17} and several other modular curves \cite{bbblmtv-19}. Similarly, before extending their strategy to give a new proof of Faltings' Theorem, Lawrence and Venkatesh first gave a new proof of the finiteness of $E_{K,S}$ by similar methods \cite{lawrence-venkatesh-18}.

Several authors have also developed implementations based on classical Chabauty and non-abelian Chabauty to compute the set $E_{K,S}$ for some $[K:\bbQ]$ and $\# S$ small. One of the first such works was \cite{dan-cohen-wewers-15}.  In more recent work, \cite{best-et-al-21} reports on their recently developed Sage code (building on work of \cite{betts-dogra-20}) that uses congruence conditions from the primes in $S$ to compute some $E_{K,S}$ when $S$ is larger than a na\"ive implementation of quadratic nonabelian Chabauty would allow. In the context of the classical Chabauty's method, the rough idea is that if $x,y \in \cO_{K,S}^{\times}$ satisfy $x+y = 1$, then for each prime $\fkp \in S$, either $v_{\fkp}(x) = v_{\fkp}(y)$ or $v_{\fkp}(x) = 0$ or $v_{\fkp}(y) = 0$. So, one can use subgroups of $J(\cO_{K,S})$ in place of the full generalized Jacobian. An efficient implementation of the algorithm proposed in the present article would use these refinements to choose a smaller $q$ in Algorithm~\ref{algorithm}. Unfortunately, their refinements only mitigate the effect of primes in $S$ and do not remove the need for descent, nonabelian Chabauty, or other arguments beyond RoS Chabauty. So, for the sake of clarity of exposition, we will not discuss these refinements further in this article. 

Our (perhaps optimistic) dream is that studying the analogue of the ideas developed in this work in the proper case might one day play a role in yet another effective proof of Faltings' Theorem. We believe that if a proof based on RoS Chabauty and descent is every discovered, then it may be suitable for practical computation. Unfortunately, such a proof seems far beyond our current technology. At a minimum, such a proof would require a better understanding of how ranks of Jacobians grow in towers of \'etale covers, how subabelian varieties of abelian varieties with repeated isogeny factors intersect other subvarieties, and how $p$-adic analytic subgroup schemes interest algebraic subvarieties. These topics are beyond the scope of this article and the needed knowledge is well beyond current mathematical understanding. 

\subsection{Notation and Conventions}

Throughout, we use the following notation and conventions:
\begin{itemize}
	\item $\bbQ \subset K' \subset K$ are number fields.
	\item $d = [K: \bbQ]$ and $d' = [K':\bbQ]$.
	\item $S_{0}$ is a finite set of finite places of $\bbQ$.
	\item $S$ is the set of places of $K$ lying above $S_{0}$ and $S'$ is the set of places of $K'$ lying above $S_{0}$.
	\item $\Sigma_{\infty}$ is the set of infinite places of $K$. 
	\item If $\fkp$ is a finite place of $X$ then $G_{\fkp}$ is the decomposition group at $\fkp$. If $\fkp \in \Sigma_{\infty}$, then $G_{\fkp}$ is either the trivial group if $\fkp$ is a complex place and is isomorphic to $\bbZ/2\bbZ$ (with non-trivial element corresponding to complex conjugation) if $\fkp$ is a real place. 
	\item $R = \cO_{K,S}$ and $R' = \cO_{K', S'}$ and $R_{0} = \cO_{\bbQ, S_{0}}$ are the rings of $S, S',$ and $S_{0}$-integers, respectively. Note that conventions around $S$-integers are not entirely standardized. While we assume $S$ consists only of finite places, other articles may require that $S$ contains all infinite places implicitly in this notation.
	\item Given a prime $\fkp$ of $R$, the completion of $K$ at $\fkp$ is $K_{\fkp}$, the ring of integers in $K_{\fkp}$ is $R_{\fkp}$, and the residue field is $k_{\fkp}$.
	\item For $F$ a number field, $r_{1}(F)$ is the number of real embeddings of $F$ and $r_{2}(F)$ is the number of complex embeddings of $F$. 
	\item Given a divisor $\fkm$ on a curve $C/K$, we write $\supp(\fkm)$ for the support of $\fkm$.
	\item The dimension of a flat scheme over $R$ or $R_0$ refers to the \emph{relative dimension} of the scheme, or equivalently the dimension of the generic fiber.
	\item If $P$ is a $\overline{\bbQ}$-point of a curve defined over a number field $K$, then $K(P)$ is the minimal field of definition of $P$. 
	\item A CM field is a totally complex number field which is a quadratic extension of a totally real number field. 
	\item  RoS Chabauty is shorthand for Chabauty's method applied to a restriction of scalars of a curve $X$ equipped with a map to the corresponding restriction of scalars of the generalized Jacobian $J$ of $X$.
	
\end{itemize}

We also use the following conventions:

Given a set of algebraic $S$-integers $\{x_{1}, \dots, x_{n}\}$ which is stable under the action of $\Gal(\overline{K}/K)$, let $f \in R[x]$ be the monic square-free polynomial with roots $x_{1}, \dots, x_{r}$. Define 
\[
\bbP^{1}_{R} \smallsetminus (\{x_1, \dots, x_{r} \} \cup \{\infty\}) \colonequals \Spec R[x, f(x)^{-1}]\,.
\]
If the $x_{i}$ are $S$-units, we define $\bbP^{1}_{R} \smallsetminus \{x_1, \dots, x_{r} \}$ by gluing $\bbP^{1}_{R} \smallsetminus \{x_1, \dots, x_{r}, \infty \}$ and $\bbP^{1}_{R} \smallsetminus \{x_1^{-1}, \dots, x_{r}^{-1}, \infty \}$ by identifying the subsets $\bbP^{1}_{R} \smallsetminus \{0, x_1, \dots, x_{r}, \infty \}$ and $\bbP^{1}_{R} \smallsetminus \{0, x_1^{-1}, \dots, x_{r}^{-1}, \infty \}$ via the morphism $t \mapsto t^{-1}$.

\subsection{Acknowledgements}

Thank you to Nils Bruin, Pete Clark, K{\'a}lm{\'a}n Gy{\H{o}}ry, Dino Lorenzini, Bjorn Poonen, Padmavathi Srinivasan, and to the anonymous referee for helpful feedback on drafts of this work.

\section{Background on RoS Chabauty} \label{sec:setup}

\subsection{Generalized Jacobians of punctured curves} \label{sec:prelim-defs}

In order to extend RoS Chabauty to affine curves, we need a notion of the generalized Jacobian of a curve $X$ over $R$. It will be helpful to slightly restrict the relative curves $X$ that we allow.

Our definitions are chosen so that 
\begin{enumerate}
	\item $J$ is a commutative group scheme of finite type over $R$, so that $J(R)$ is a finitely-generated abelian group;
	\item The generic fiber $J_{K}$ of $J$ is the generalized Jacobian of the generic fiber $X_{K}$ of $X$; 
	\item An Abel-Jacobi map from the generic fiber $X_{K}$ to its generalized Jacobian $J_{K}$ extends to a morphism from $X$ to $J$ whenever it is defined with respect to a divisor of $X_{K}$ which extends to a divisor of $X$. 
\end{enumerate}

\begin{definition}\label{def:models}
	Let $X$ be a flat relative curve over finite-type over $R$ which is smooth, separated, and has geometrically connected fibers. Let $C$ be a smooth proper curve over $K$ equipped with a reduced effective divisor $\fkm$. Given any proper regular model $\cC$ for $C$, let $\Gamma = \overline{\supp(\fkm)}$ be the closure of $\fkm$ in $\cC$ in the Zariski topology.
	
	We say that $X$ is a regular model for $(C,\fkm)$ if there is some proper regular model $\cC/R$ of $C$ such that $X \cong \cC \smallsetminus \Gamma$.
\end{definition}

\begin{remark}
	There may exist pairs $(C,\fkm)$ of a curve and reduced effective divisor with no regular models, since removing the closure of $\fkm$ may disconnect finitely many fibers of a regular model of $C$. However, after possibly enlarging $S$ by a finite set of primes, the pair $(C,\fkm)$ will have a regular model. When applying Chabauty's method, we generally prefer to take the smallest such enlargement of $S$.
\end{remark}

Regular models $X/R$ are well-adapted to Chabauty's method because they are equipped with an $R$-embedding into their \emph{generalized Jacobian}, a commutative group scheme $J_{X}$ of finite type over $R$ with the property that $J_{X}(R)$ is a finitely-generated abelian group. See Chapter 5 of \cite{serre-88} for the theory of generalized Jacobians of curves over number fields. We now describe what we mean by the generalized Jacobian of a regular model of $(C,\fkm)$.

If $X$ is a regular model of $(C,\fkm)$, its generic fiber $X_{K}$ has the form $C \smallsetminus \supp(\fkm)$. The generalized Jacobian $J_{C,\fkm}$ of $C$ with reduced modulus $\fkm$ is a semi-abelian group scheme (an extension of an abelian variety by a torus) over $K$. Let $\cJ_{X}$ be the (locally of finite type) N\'{e}ron model for $J_{C,\fkm}$ and let $J_{X}$
be the connected component of the identity of $\cJ_{X}$. (See \cite{bosch-90} \S~10.2 Theorem~2.) We abuse notation slightly and call $J_{X}$ the \emph{generalized Jacobian} of $X$. 

\begin{lemma}
	If $X/R$ is a regular model of $(C,\fkm)$, then $J_{X}(R)$ is a finitely-generated abelian group.
\end{lemma}

\begin{proof}
	The semi-abelian variety $J_{X_{K}}$ is an extension of an abelian variety $A$ by a torus $T$. The group structure on $J_{X_{K}}(K)$ is abelian, so the same is true for $J_{X}(R)$.
	
	(Part 1.) We first consider the case where $T \cong \bbG_{m,K}^t$ is a split torus. Let $\cA$ be the N\'eron model of $A$.
	From the proof of \cite{bosch-90} \S10.1 Prop.~7, $J_{X}$ is an extension of $\cA^{\circ}$ by $\bbG_{m,R}^{t}$. So there is an exact sequence
	\[
	0 \to \bbG_{m,R}^{t}(R) \to J_{X}(R) \to \cA^{\circ}(R) \to 0\,.
	\]
	Now, $\cA^{\circ}(R) \subset A(K)$ is finitely-generated by the Mordell-Weil theorem and $\bbG_{m,R}^{t}(R) = (R^{\times})^t$ is finitely-generated by Dirichlet's unit theorem, so $J_{X}(R)$ is finitely-generated in this case.
	
	(Part 2.) Now suppose $T$ is a non-split torus. Let $L$ be a finite extension of $K$ where $T$ splits. Let $S_{L}$ be a set of places of $L$ lying above $S$ together with any primes which ramify in $L/K$. Let $R_{L}$ be the set of $S_{L}$-integers in $L$. Taking l.f.t N\'eron models commutes with \'etale base change (\cite{bosch-90} \S10.1 Prop.~3) so $(\mathcal J_{X})_{L} = \mathcal J_{X_{L}}$. The identity component of each fiber of $\mathcal J_{X}$ is smooth, connected, and contains a rational point (the identity element ), so each fiber of $\mathcal J_{X}$ is geometrically connected. In particular, $J_{X_{R_{L}}} = (J_{X})_{R_{L}}$. This allows us to identify $J_{X}(R)$ as a subgroup of $J_{X_{R_{L}}}(R_{L})$, which is a finitely-generated abelian group by Part 1. Hence, $J_{X}(R)$ is a finitely-generated abelian group.
\end{proof}

Suppose we are given a point $P \in X(R)$. Then $X_{K}$ is equipped with an Abel-Jacobi $K$-morphism $j_{K}: X_{K} \hookrightarrow J_{X_{K}}$ which sends $P_{K}$ to the identity element of $J_{X_{K}}$.\footnote{If $X_K$ is proper, this is standard. If $X_K$ is affine, \cite{serre-88} describes a $K$-morphism $X_K \hookrightarrow J_{X_{K}}$. One can ensure that $P_{K}$ maps to the identity element by translating by $-1$ times the image of $P_{K}$.}
By the N\'{e}ron mapping property, $j_{K}$ extends to a map $j: X \hookrightarrow \cJ_{X}$ which sends $P$ to the identity section of the group scheme $J_{X}$. Since the fibers of $X$ are geometrically connected and their image includes the identity, we see that the image of $X$ under $j$ lies in $J_{X}$. 

In the absence of an integral point $P \in X(R)$, we can define a finite $R$-morphism $j: X \to J_{X}$ with respect to any horizontal divisor of $X$. By a mild abuse of the language we will call any such map an Abel-Jacobi map.

\subsubsection{Generalized Jacobians of genus $0$ curves} \label{sec:gen-jac-g0}

For the remainder of this article, we will primarily be concerned with the case where $X$ is regular model of $(C,\fkm)$ where $C$ has genus zero. In preparation, we recall some facts about the structure of the generalized Jacobian of such a curve. 

Suppose $C/K$ is genus $0$, that $\fkm$ is a reduced effective divisor on $C$ and let $X/R$ be a regular model for $(C,\fkm)$. Note that we do not assume that $C$ has a $K$-point, so $C$ may not be isomorphic to $\bbP^1_{K}$. Let $P_{1}, \dots, P_{c}$ be distinct irreducible divisors so that $\fkm = P_1 + \cdots + P_{c}$. For $i \in \{1, \dots, c\}$, let $L_{i}$ be the residue field of $P_{i}$, let $S_{i}$ be the set of places of $L_{i}$ which lie over $S$, and let $R_{i}$ be the ring of $S_{i}$-integers in $L_{i}$. Then, the generalized Jacobian of $C$ with modulus $\fkm$ is 
\begin{align} \label{eqn:jacobian-punctured-P1-structure-over-K}
J_{C,\fkm} \cong \left(\prod_{i=1}^{c} \Res_{L_i/K} \bbG_{m, L_i}\right)/\Delta(\bbG_{m, K})\,,
\end{align}
where the $\Delta(\bbG_{m, K})$ denotes the diagonal embedding of $\bbG_{m,K}$ into the product of the Weil restrictions. The generalized Jacobian of $X$ is the connected component of the lft-N\'eron model of $J_{C,\fkm}$. Using Proposition~4.2 of \cite{liu-lorenzini-01} to compute the quotient, this is
\begin{align} \label{eqn:jacobian-punctured-P1-structure}
J_{X} \cong \left(\prod_{i=1}^{c} \Res_{R_{i}/\OKS} \bbG_{m, R_{i}}\right)/\Delta(\bbG_{m, \OKS})\,,
\end{align}
where $\Delta(\bbG_{m, \OKS})$ denotes the diagonally embedded copy of $\bbG_{m, \OKS}$. If $L_{i} = K$ for some $i$, the formula can be simplified by leaving out the $i$th component and the quotient. 

The expression (\ref{eqn:jacobian-punctured-P1-structure}) makes it easy to compute the dimension of $J_X$ and rank of $J_X(\OKS)$. 
Since $R_{i} = \cO_{L_{i},S_{i}}$, we have
\[
\dim J_{X}  = \sum_{i = 1}^{c} [L_i:K] - [K:K] 
\]
and 
\begin{align} \label{eqn:jacobian-punctured-P1-rank}
\rank J_{X}(\OKS) 
= & \sum_{i =1}^{c} \rank \cO_{L_i,S_i}^\times - \rank \OKS^\times  \\
= & \sum_{i = 1}^{c} [r_1(L_i) + r_2(L_i) + \# S_{i} - 1] - [r_1(K) + r_2(K) + \#S - 1]\,. \nonumber
\end{align}

We can also express the rank in terms of the action of the absolute Galois group of $K$ on the set of punctures of our genus zero curve. We state the general result. 

\begin{lemma} \label{lem:rank+dim-genus-0-jac}
	Let $C/K$ be a genus $0$ curves and let $\fkm$ be a reduced effective divisor on $C$. Let $\cC/R$ be a proper regular model of $C$, let $\Gamma = \overline{\supp(\fkm)}$ and let $X = \cC/\Gamma$ be a regular model of $(C,\fkm)$. View $\Gamma(\overline{K}) =  \supp(\fkm)$as a set of $\overline{K}$-points on the generic fiber $X_{K}$. Let $G = \Gal(\overline{K}/K)$ be the absolute Galois group of $K$. Let $\Sigma_{\infty}$ be the set of infinite places of $K$. For each finite place $\fkp$, let $G_{\fkp}$ denote the decomposition group of $\fkp$. For $\fkp \in \Sigma_{\infty}$, let $G_{\fkp}$ be the trivial group if $\fkp$ is a complex place and $\bbZ/2\bbZ$ (corresponding to complex conjugation) if $K$ is a real place.
	
	Then, the generic fiber $J_{C,\fkm}$ of $J_{X}$ is a torus,  
	\begin{align*}
	\dim J_{X} = \# \Gamma(\overline{K}) - 1,
	\end{align*}
	and $J_{X}(\OKS)$ is an abelian group of rank
	\begin{align*}
	\rank J_{X}(\OKS) = &  - \left(\# (G\backslash \Gamma(\overline{K})) - 1\right) + \sum_{\fkp \in S \cup \Sigma_{\infty}} \left(\# (G_{\fkp} \backslash \Gamma(\overline{K})) - 1\right)\,.
	\end{align*}
\end{lemma}

{ 
	We omit the proof, which is a fairly straightforward computation in the character theory of tori. The rank computation can be found, for example, Theorem~8.7.2 of \emph{Cohomology of Number Fields} by Neukirch, Schmidt, and Wingberg \cite{neukirch-schmidt-wingberg-08}, or in Chapter 6 of Eisentr\"ager's Ph.D. Thesis \cite{eisentrager-03}.
}

When $X$ is a subscheme of $\bbA^1_{R}$, it is possible to define an Abel-Jacobi map $j: X \to J$ without reference to an integral point on $X$. We give two examples since we will not need this fact going forward and the general case is notationally burdensome.

\begin{example}
	If $X = \bbP^1_{\OKS} \smallsetminus \{0,1,\infty\} = \bbA^1_{\OKS} \smallsetminus \{0,1\}$, then $J_{X_{K}} \cong \bbG_{m, K} \times \bbG_{m,K}$. Then, $J_{X} \cong \bbG_{m, \OKS} \times \bbG_{m, \OKS}\,$ and the Abel-Jacobi map is 
	\begin{align*}
	j : X & \to J_{X} \\
	x & \mapsto (x, x-1)\,,
	\end{align*}
	where we view $X \subset \mathbb A^1_{\OKS}$ and use the simplified form for the generalized Jacobian.
\end{example}

\begin{example}
	Suppose $\sqrt{2} \notin K$, set $L = K(\sqrt{2})$, and let $R'$ be the integral closure of $R$ in $L$. Suppose also that $\{1, \sqrt{2}\}$ is an integral basis for $R'$ over $R$.
	
	If $X = \bbP^1_{\OKS} \smallsetminus \{\pm \sqrt{2}, \infty\} = \bbA^1_{R} \smallsetminus \{\pm \sqrt{2} \}$, then $J_{X_{K}} \cong \Res_{L/K} \bbG_{m,L}$ and 
	\begin{align*}
	J_{X} & \cong \Res_{R'/R} \bbG_{m, R'} \\
	& = \Spec R[x_1, x_2, y_1, y_2]/(x_1y_1 + 2 x_2 y_2 - 1, x_1y_2 + x_2 y_1)\,.
	\end{align*}
	Then, we have the Abel-Jacobi map
	\begin{align*}
	j : X & \to J_{X} \\
	x & \mapsto \left(x, -1, \frac{x}{x^2 - 2}, \frac{1}{x^2 - 2}\right) ``= x - \sqrt{2}\,. \text{''}
	\end{align*}
\end{example}

\subsection{Chabauty's method for curves} \label{sec:classical-chabauty}

Before discussing the higher-dimensional case, we recall Chabuaty's $p$-adic method for computing the set of $\OKS$-points on a regular model $X$ of some curve/reduced effective divisor pair $(C,\fkm)$. Let $J$ be the generalized Jacobian of $X$. In modern language, Chabauty and Skolem proved that if $\rank J(\OKS) \leq \dim J - 1$, then $X(\OKS)$ is finite. We now summarize their argument.

Suppose there is some $P_{0} \in X(R)$. (Otherwise, $X(R)$ is empty and therefore finite.) Let $j: X \to J$ be the Abel-Jacobi map with respect to $P_{0}$. Set $g = \dim J$.
Fix a prime $\fkp$ of $R$ of good reduction for $X$. Say $\fkp$ lies over the rational prime $p$.  Then, $J(K_{\fkp})$ in a $p$-adic Lie group, so it is equipped with a \emph{logarithm} map to its Lie algebra, which is isomorphic to $K_{\fkp}^{g}$. Concretely, if we fix a basis $(\omega_{1}, \dots, \omega_{g})$ for $H^{0}(J_{K}, \Omega^1)$, the map is given in coordinates by the abelian integrals
\begin{align*}
\log: J(\cO_{K_{\fkp}}) &    \to      K_{\fkp}^{g} \,, \\
Q     &  \mapsto    \left( \int_{O}^{Q} \omega_{i} \right)_{i} \,.
\end{align*}
Let $k_{\fkp}$ be the residue field at $\fkp$. The kernel of $\log$ is isomorphic to $J(k_{\fkp})$. In particular, the fibers are finite. There is a commutative diagram
\begin{equation}\label{eqn:Chabauty-diagram}
\begin{tikzcd} 
X(\OKS) \arrow[hook]{r} \arrow{d}{j} & X(\cO_{K_{\fkp}}) \arrow{d}{j}  \arrow{dr} &  \\
J(\OKS) \arrow[hook]{r}  & J(\cO_{K_{\fkp}}) \arrow{r}{\log} & K_{\fkp}^{g} \,.
\end{tikzcd}
\end{equation}
The image of $J(R)$ in $K_{\fkp}^{g}$ is contained in a $K_{\fkp}$-linear subspace of dimension at most $\rank J(R)$. On the other hand, Chabauty proved that $j(X(\cO_{K_{\fkp}}))$ is Zariski-dense in $K_{\fkp}^{g}$. 
In particular, if the classical Chabauty inequality 
\begin{align} \label{eqn:classical-Chabauty-inequality}
\rank J(\OKS) \leq \dim J - 1
\end{align}
is satisfied, then $\log X(K_{\fkp})$ is not contained in $\Span_{K_{\fkp}}\log J(\OKS)$. 
Since $\log X(\cO_{K_{\fkp}})$ is a $\fkp$-adic analytic curve, the intersection 
$\log X(\cO_{K_{\fkp}}) \cap \Span_{K_{\fkp}} \log J(R)$ in $K_{\fkp}^{g}$ is finite. Since $\log$ has finite kernel, the intersection
\begin{align}
X(\cO_{K_{\fkp}})_{S,1} \colonequals \log^{-1}\left( X(\cO_{K_{\fkp}}) \cap \Span_{K_{\fkp}} \log J(R)\right)
\end{align}
is also finite. Since $X(\OKS) \subset X(\cO_{K_{\fkp}}) \cap J(\OKS) \subset J(\cO_{K_{\fkp}})\,,$ the set $X(\OKS)$ is finite as well.

In the case of $X$ proper, Coleman \cite{coleman-85} showed how to use this strategy to show that $\# X(\OKS)$ is contained in an explicitly-computable subset of $X(\cO_{K_{\fkp}})$ of size at most $X(k_{\fkp}) + 2g - 2$. See the excellent survey paper of McCallum-Poonen \cite{mcCallum-poonen-12} for more detail.

\begin{remark}
	When $K \neq \bbQ$, the rank bound in Chabauty's method is inefficient in two ways.
	\begin{enumerate}
		\item (If $[K_{\fkp}:\bbQ_{p}] > 1$.) While $\Span_{K_{\fkp}}\log J(\OKS)$ is the smallest $K_{\fkp}$-subspace of $K_{\fkp}$ containing $J(\OKS)$, it is typically much larger than $\overline{\log J(\OKS)} = \Span_{\bbQ_{p}} \log J(\OKS)$, the saturation of the closure of $\log J(\OKS)$ in the $p$-adic topology on $K_{\fkp}^{g}$. 
		\item (If $[K_{\fkp}:\bbQ_{p}] < [K:\bbQ]$.) Using only the $\cO_{K_{\fkp}}$-points for Chabauty's method ignores all information coming from other primes above $p$. Combining Chabauty's method at all primes above $p$ has the potential to compute $X(\OKS)$ when $J(\OKS)$ has larger rank.
	\end{enumerate}
	The solution to both inefficiencies is to apply an analogue of Chabauty's method to $\Res_{K/\bbQ} X$ instead. Calculating expected dimensions suggests that Chabauty's method applies to $\Res_{K/\bbQ} X$ would suffice to compute $X(K)$ whenever $\rank J(K) \leq [K:\bbQ] \cdot (g - 1)$. Although we expect this to hold in generic examples, it is not sufficient in several particular examples.
\end{remark}

\subsection{Chabauty's method for restrictions of scalars} \label{sec:RoSChabauty}

Chabauty's method extends to a strategy which can be used to compute integral points on higher-dimensional schemes in some cases. 
Suppose that $W$ is a scheme over $\OKS_{0}$ equipped with a morphism $j$ to its Albanese $A$ (or any commutative group scheme over $\OKS_{0}$ such that $\log(j(Z(\bbZ_{p}))$ is Zariski dense in $\log(W(\bbZ_{p}))$.) By replacing $X$ with $W$ and $J$ with $A$ in \eqref{eqn:Chabauty-diagram}, we get the diagram
\begin{equation}\label{eqn:Chabauty-diagram2}
\begin{tikzcd} 
W(\OKS_{0}) \arrow[hook]{r} \arrow{d}{j} & W(\bbZ_{p}) \arrow{d}{j}  \arrow{dr} &  \\
A(\OKS_{0}) \arrow[hook]{r}  & A(\bbZ_{p}) \arrow{r}{\log} & \bbQ_{p}^{\dim A} \,.
\end{tikzcd}
\end{equation}

If the intersection of $\Span_{\bbQ_{p}} \log A(\OKS_{0})$ and $\log \circ j(X(\bbZ_{p}))$ is finite, then $W(R_{0})$ is contained in an explicitly computable finite set of fibers of $j$. When $\rank A(R_{0}) \leq \dim A_{\bbQ} - \dim W_{\bbQ}$, the expected dimension of $\log j(W) \cap \Span_{\bbQ_{p}}  \log A(\OKS_{0})$ is zero, so we `expect' that this intersection is finite. 

\begin{remark}
	Since we are now working over subrings of $\bbZ_{p}$, we can simplify notation somewhat by working directly in the Jacobian rather than in its Lie algebra. Let $\overline{A(\OKS_{0})}$ denote the saturation of the closure of $A(\OKS_{0})$ in $A(\bbZ_{p})$ in the $p$-adic topology. Equivalently, $\overline{A(\OKS_{0})}$ is the preimage of $\Span_{\bbQ_{p}}  \log A(\OKS_{0})$ in $A(\bbZ_{p})$. Then, $\log j(W) \cap \Span_{\bbQ_{p}}  \log A(\OKS_{0})$ is finite if and only if $j(W)\cap \overline{A(\OKS_{0})}$ is finite. For compactness of notation, we switch to this perspective going forward.
\end{remark}

We can use this strategy to compute $X(R)$ by taking $W = \Res_{\OKS/R_{0}} X$ and $A = \Res_{\OKS/R_{0}} J$. We have
\begin{equation}\label{eqn:Chabauty-diagram3}
\begin{tikzcd} 
X(\OKS) \arrow[equal]{r} \arrow{d}{j} &  (\Res_{\OKS/R_{0}} X)(R_0) \arrow[hook]{r} \arrow{d}{j} & (\Res_{\OKS/R_{0}} X)(\bbZ_{p}) \arrow{d}{j}  \\
J(\OKS) \arrow[equal]{r} & (\Res_{\OKS/R_{0}} J) (R_0) \arrow[hook]{r}  & (\Res_{\OKS/R_{0}} J)(\bbZ_{p})  \,.
\end{tikzcd}
\end{equation}
In this case, we `expect' the intersection 
\begin{align}\label{eqn:RoS-Chabauty-set}
X(\cO_{K} \otimes \bbZ_{p})_{S,1} \colonequals j(\Res_{\OKS/R_{0}} X)(\bbZ_{p}) \cap \overline{(\Res_{\OKS/R_{0}} J)(R_0)} 
\end{align}
will be finite when the RoS Chabauty inequality
\begin{align} \label{eqn:RoS-Chabauty-inequality}
\rank J(\OKS) \leq d(\dim J - 1)
\end{align}
holds. We call this strategy \emph{Chabauty's method for restrictions of scalars} or \emph{RoS Chabauty} for short. Since $X(R) \subset X(\cO_{K} \otimes \bbZ_{p})_{S,1}$, one can hope to compute $X(R)$ using RoS Chabuaty and a Mordell-Weil sieve if $X(\cO_{K} \otimes \bbZ_{p})_{S,1}$ is finite. 

\cite{dogra-19} shows that $X(R)$ is not Zariski-dense in $(\Res_{\OKS/R_{0}} X)(\bbZ_{p})$ when \eqref{eqn:RoS-Chabauty-inequality} holds. Unfortunately, this is no longer enough to conclude that $X(\cO_{K} \otimes \bbZ_{p})_{S,1}$ is finite when $\dim W = d > 1$. 

In fact, $X(\cO_{K} \otimes \bbZ_{p})_{S,1}$ can be infinite even when \eqref{eqn:RoS-Chabauty-inequality} is satisfied! Let $T \subset A$ be a subgroup scheme, $x \in A(\OKS_{0})$ be an integral point, and let $B = x \cdot T$ be the translate of $T$ by $x$. We have
\begin{align} \label{eqn:bad-subset}
(B \cap j(W))(\bbZ_p) \cap \overline{T(R_{0})}  \subseteq j(W)(\bbZ_p) \cap \overline{A(\OKS_{0})} \,.
\end{align}
If also
\begin{align} \label{eqn:bad-inequality}
\dim \overline{T(\OKS_{0})} \geq \dim B - \dim(j(W) \cap B)\,
\end{align}
then the left-hand side of \eqref{eqn:bad-subset} will be infinite, so the right-hand side will as well. In particular, the set $X(\cO_{K} \otimes \bbZ_{p})_{S,1}$ of points cut out by RoS Chabuaty for $X$ will be infinite.

The point is that even when $\rank J(R)$ is small, RoS Chabauty may not produce a finite set of $p$-adic points if the rank is concentrated in a subgroup scheme of $\Res_{\OKS/\OKS_{0}} J$ which intersects $\Res_{\OKS/\OKS_{0}} X$ with larger than expected dimension. 

Proposition~\ref{prop:base-change-obstruction}, due to \cite{wetherell-00} and first appearing in print in \cite{siksek-13} explains one source of translates $B$ of subgroup schemes of $\Res_{\OKS/\OKS_{0}} J$ with $\dim(B \cap j(\Res_{\OKS/\OKS_{0}} X))$ larger than expected which can cause $X(\cO_{K} \otimes \bbZ_{p})_{S,1}$ to be infinite.

Recall that by convention, $K'$ is a subfield of $K$ and $R' = K' \cap R$ is the set of $S'$ integers in $K'$ for $S'$ the set of primes lying under $S$.

\begin{proposition}\label{prop:base-change-obstruction}
	Let $C/K'$ be a curve equipped with a reduced effective divisor $\fkm$. Suppose that $X/\OKSp$ is a regular model of $(C,\fkm)$ and that $X_{\OKS}$ is also a regular model of $(C_{K}, \fkm_K)$. If $X(\cO_{K'} \otimes \bbZ_{p})_{S',1}$ is infinite, then RoS $X(\cO_{K} \otimes \bbZ_{p})_{S,1}$ is also infinite.
\end{proposition}

\begin{proof}
	The maps
	\begin{align*}
	j: & \Res_{R'/R_{0}} X \to \Res_{R'/R_{0}} J \\
	j: & \Res_{R/R_{0}} X_{R} \to \Res_{R/\bbQ} J_{R} 
	\end{align*}
	induced by the Abel-Jacobi map $j: X \to J$ commute with the diagonal embeddings 
	\begin{align*}
	\iota: &  \Res_{R'/R_{0}} X \hookrightarrow \Res_{R/R_{0}} X_{R} \\
	\iota: &\Res_{R'/R_{0}} J \hookrightarrow \Res_{R/R_{0}} J_{R} 
	\end{align*}
	In particular, the embeddings $\iota$ induce
	\begin{align*}
	  j(\Res_{\OKS/\OKS_{0}} X)(\bbZ_{p}) \cap \overline{(\Res_{\OKS/\OKS_{0}} J) (R_0)}\subseteq  j(\Res_{\OKS/\OKS_{0}} X_{\OKS})(\bbZ_{p}) \cap \overline{(\Res_{\OKS/\OKS_{0}} J_{\OKS}) (R_0)}\,.
	\end{align*}
	The left-hand side is infinite by assumption, so the right-hand side is infinite as well.
\end{proof}

\begin{remark}
	It is easy to find examples of curves satisfying the conditions of Proposition~\ref{prop:base-change-obstruction} by brute-force search, first for smooth proper curves $X/\bbQ$  with $\rank J(\bbQ) < \text{genus}(X)$ and then for number fields $K$ such that $\rank J(K)$ is not much larger than $\rank J(\bbQ)$. See \cite{siksek-13} for an explicit example.
	
	However, Proposition~\ref{prop:base-change-obstruction} does not explain all cases where $X(\cO_{K} \otimes \bbZ_{p})_{S,1}$ is infinite.
\end{remark}

\begin{proposition}\label{prop:full-prym-obstruction}
	Let $C, D/K$ be smooth proper curves, let $\fkm_{C}, \fkm_{D}$ be reduced effective divisors on $C$ and $D$, and let $X,Y/\OKS$ be regular models for $(C,\fkm_{C})$ and $(D,\fkm_{D})$, respectively. 
	
	Let $f: X \to Y$ be a morphism and let $f_{*}: \Res_{\OKS/R_{0}} J_{X} \to \Res_{\OKS/R_{0}} J_{Y}$ be the induced pushforward map between the restrictions of scalars of the Jacobians of $X$ and $Y$. Let $P = f_{*}^{-1}(O)$ be the preimage of the origin on $J_{Y}$. ($P$ is a generalized Prym variety associated to $f$.)
	
	Suppose that $\overline{P(R_{0})} = P(\bbQ_{p})$. Suppose also that there is some $x \in (\Res_{\OKS/R_{0}} X)(\bbQ_{p})$ such that the image $f_{*}(j(x))$ lies in $Y(\cO_{K} \otimes \bbZ_{p})_{S,1} =  j(\Res_{\OKS/R_{0}} Y_{\OKS})(\bbZ_{p}) \cap \overline{(\Res_{\OKS/R_{0}} J_{Y}) (R_0)} $ and is not isolated in $Y(\cO_{K} \otimes \bbZ_{p})_{S,1}$ in the $p$-adic topology. (In particular, $Y(\cO_{K} \otimes \bbZ_{p})_{S,1}$ is infinite.) Then, $X(\cO_{K} \otimes \bbZ_{p})_{S,1}$ is infinite.
\end{proposition}

\begin{proof}
	First, we claim that 
	\begin{align} \label{eqn:full-fibers}
	f_{*}^{-1}(\overline{J_{Y}(R_{0})}) \subset \overline{J_{X}(R_{0})}\,.
	\end{align} Indeed, for any $y \in \overline{J_{Y}(R_{0})}$, we have $(f_{*}^{-1}(y))(\bbZ_{p}) = f^{*}(y)\cdot P(\bbZ_{p})$. Claim \eqref{eqn:full-fibers} follows since $f^{*}(y) \in \overline{J_{X}(R_{0})}$ and $P(\bbZ_{p}) = \overline{P(R_{0})} \subset \overline{J_{X}(R_{0})}$.
	
	Now, \eqref{eqn:full-fibers} implies that if $x' \in (\Res_{\OKS/R_{0}} X)(\bbZ_{p})$ and $j(f(x')) \in I_{Y}$, then 
	\[
	j(x') \in X(\cO_{K} \otimes \bbZ_{p})_{S,1} = j(\Res_{\OKS/R_{0}} X_{\OKS})(\bbZ_{p}) \cap \overline{(\Res_{\OKS/R_{0}} J_{X}) (R_0)}\,.
	\] 
	To conclude, we observe by Krasner's Lemma that if $y \in Y(\cO_{K} \otimes \bbZ_{p})_{S,1}$ is sufficiently close $p$-adically to $j(f(x)) = f_{*}(j(x))$, then there is some $x' \in (\Res_{\OKS/R_{0}} X)(\bbZ_{p})$ such that $y = j(f(x'))$. In particular, $j(x') \in X(\cO_{K} \otimes \bbZ_{p})_{S,1}$, so $X(\cO_{K} \otimes \bbZ_{p})_{S,1}$ is infinite.
\end{proof}

\begin{example} \label{ex:dogra} 
	In \cite{dogra-19}, Dogra gives an example where both \eqref{eqn:RoS-Chabauty-inequality} and the conditions of Proposition~\ref{prop:full-prym-obstruction} are satisfied (at least up to finite index.) In Dogra's example, $K$ is a quadratic field,  the curve $Y/\bbQ$ has genus $2$ with $\rank J_{Y}(\bbQ) = \rank J_{Y}(K) = 2$, the curve $X/K$ has genus $3$ with $\rank J_{X}(K) = 4$, and $f: X \to Y_{K}$ is an unramified morphism of curves.
	
	In Dogra's example, $\overline{J_{Y}(\bbQ)} = J_{Y}(\bbQ_{p})$, so $Y(\bbQ \otimes \bbQ_{p})_{1}$ is infinite. By Proposition~\ref{prop:base-change-obstruction}, the set $Y(K \otimes \bbQ_{p})_{1}$ is also infinite.
	
	Let $P = f_{*}^{-1}(O)$ as in Proposition~\ref{prop:full-prym-obstruction}. Then,
	\[
	\rank P(K) = \rank J_{X}(K) - \rank J_{Y}(K) = 4 - 2 = 2 = \dim P\,.
	\]
	With a bit more work, one can check that in Dogra's example, $\overline{P(R_{0})} = P(\bbQ_{p})$ and there is some $x \in (\Res_{\OKS/R_{0}} X)(\bbQ_{p})$ such that the image $f_{*}(j(x))$ in $Y(K \otimes \bbQ_{p})_{S,1}$ is not isolated. By Proposition~\ref{prop:full-prym-obstruction}, the set $X(K \otimes \bbQ_{p})_{1}$ of points cut out by RoS Chabauty applied to $X$ is infinite as well.
	
	On the other hand, 
	\begin{align}
	\rank J_{X}(K) = 4  \leq 4 = 2(3 - 1) = [K:\bbQ] (\dim J - 1)\,,
	\end{align}
	so the RoS Chabauty inequality \eqref{eqn:RoS-Chabauty-inequality} is satisfied. Moreover, Dogra's $X$ is not the base change of any curve defined over $\bbQ$, so the infinitude of $X(K \otimes \bbQ_{p})_{1}$ cannot be explained by Proposition~\ref{prop:base-change-obstruction} alone.
\end{example}

\begin{remark}\label{rem:full-prym-obstruction}
	In the setup of Proposition~\ref{prop:full-prym-obstruction} and its proof, it is possible for the geometric structure to cause $X(\cO_{K} \otimes \bbZ_{p})_{S,1}$ to be infinite even if $\overline{P(\OKS_{0})}$ is not finite index in $P(\bbZ_{p})$. In fact, an expected dimension calculation suggests that in `typical' cases
	\[
	\dim X(\cO_{K} \otimes \bbZ_{p})_{S,1} \geq Y(\cO_{K} \otimes \bbZ_{p})_{S,1} - (\dim P - \dim \overline{P(R_{0})})\,
	\]
	so long as the preimage of some component of $Y(\cO_{K}\otimes \bbZ_{p})_{S,1}$ (as a na\"ive $p$-adic analytic variety in the sense of Serre) of maximal dimension has nonempty preimage under $f_{*}$ in $(\Res_{\OKS/\OKS_{0}} X)(\bbZ_{p})$.	
\end{remark}

\section{Main Definitions and Statement of Main Results}\label{sec:main-defs}

For simplicity of notation, we now restrict to the case that $X/R$ is a regular model of $(C,\fkm)$ where $C/K$ has genus $0$. 

In Propositions~\ref{prop:base-change-obstruction}~and~\ref{prop:full-prym-obstruction}, the infinitude of $X(\cO_{K} \otimes \bbZ_{p})_{S,1}$ is explained by the existence of a subgroup scheme $T \subset \Res_{R/R_{0}} J$ such that $\rank T(\OKS_{0})$ is large and and $T$ intersects $\Res_{R/R_{0}} X$ in higher-than-expected dimension. We make the following definition to capture this obstruction. 
  
\begin{definition}\label{def:subgroup-obstruction}
	Let $C/K$ be a genus $0$ curve, let $\fkm$ be a reduced effective divisor on $C$ and let $X$ be a regular model of $(C,\fkm)$. A \emph{subgroup obstruction} to RoS Chabauty for $X$ is a translate $x \cdot T$ of a subtorus $T \subset \Res_{R/R_{0}} J_{X}$ such that
  	\[
  	\dim( (x\cdot T) \cap j( \Res_{R/R_{0}} X ) ) > \dim T - \dim \overline{T(R_0)}\,.
  	\]
\end{definition}

All constructions of subgroup obstructions that have appeared in the literature arise from geometric constructions involving base change and direct sum with Prym varieties, as in  Propositions~\ref{prop:base-change-obstruction}~and~\ref{prop:full-prym-obstruction}. Such subgroups are especially likely to be subgroup obstructions, because they intersect $j( \Res_{R/R_{0}} X ) $ in higher-than-expected dimension for reasons explained by the geometry of the \emph{curve} $X$. Motivated by these propositions, we now define the class of BCP-subtori of $\Res_{R/R_{0}} J$ for $J$ the generalized Jacobian of regular model $X$ of $(C,\fkm)$ for some genus $0$ curve $C/K$.  As Definitions~\ref{def:bc-successor},~\ref{def:p-successor},~and~\ref{def:bcp-torus} make precise, BCP subtori are the subtori which  which can be built inductively by iteratively (1) taking diagonal embeddings from curves defined over subrings which become isomorphic to $X$ after base change to $R$ (Proposition~\ref{prop:base-change-obstruction}) and (2) adding Prym varieties (Proposition~\ref{prop:full-prym-obstruction} and Remark~\ref{rem:full-prym-obstruction}).
By replacing each instance of `subtorus' with `subgroup scheme,' these definitions generalize in a natural way to restrictions of scalars of generalized Jacobians of higher genus curves. Recall that by convention, $K'$ is a subfield of $K$. 

\begin{definition}\label{def:bc-successor}
	Let $D/K'$ be a genus $0$ curve, let $\fkm$ be a reduced effective divisor on $D$ and let $Y$ be a regular model of $(D,\fkm)$. Let $T' \subset \Res_{R'/R_{0}} J_{Y}$ be a subtorus of the restriction of scalars of the generalized Jacobian of $Y$. Suppose that $X \cong Y_{R}$ is the base change of $Y$ to $R$. The isomorphism induces a natural inclusion 
	\[
	\Res_{R'/R_{0}} J_{Y} \hookrightarrow \Res_{R/R_{0}} J_{X}
	\]
	of restrictions of scalars of generalized Jacobians. Let $T \subset \Res_{R/R_{0}} J_{X}$ be the image of $T'$ under this inclusion. We call the pair $(X,T)$ a \emph{BC-successor} of the pair $(Y,T')$.
\end{definition}

\begin{definition}\label{def:p-successor}
	Let $C,D/K$ be genus $0$ curves, let $\fkm_{C}$ and $\fkm_{D}$ be reduced effective divisors on $C$ and $D$ respectively, and let $X$ and $Y$ be regular models of $(C,\fkm_{C})$ and $(D, \fkm_{D})$, respectively. 
	
	Let $f: X \to Y$ be a morphism and let $T' \subset \Res_{R/R_{0}} J_{Y}$ be a subtorus of the restriction of scalars of the generalized Jacobian of $Y$. The map $f$ induces a push-forward map 
	\[
	f_{*} : \Res_{R/R_{0}} J_{X} \to \Res_{R/R_{0}} J_{Y}\,.
	\]
	The connected component\, $T = f_{*}^{-1}(T')^{\circ}$ of the identity of the preimage of $T'$ under $f_{*}$ is a subtorus of\, $\Res_{R/R_{0}} J_{X}$. We call the pair $(X,T)$ a \emph{P-successor} of $(Y,T')$.
\end{definition}

\begin{definition}\label{def:bcp-torus}
	Let $C/K$ be a genus $0$ curve, let $\fkm$ be a reduced effective divisor on $C$ and let $X$ be a regular model of $(C,\fkm)$.
	
	We say that a torus $T \subset \Res_{R/R_{0}} J_{X}$ is a \emph{$0$-BCP torus for $X$} if\, $T = \Res_{R/R_{0}} J_{X}$. 
	
	For $n \geq 1$, we say $T \subset \Res_{R/R_{0}} J_{X}$ is an \emph{$n$-BCP torus for $X$} if the pair $(X,T)$ is either a BC-successor or a P-successor of some pair $(Y,T')$ where $T'$ is an $(n-1)$-BCP torus for $Y$.
	
	We say $T \subset \Res_{R/R_{0}} J_{X}$ is a \emph{BCP torus for $X$} if $T$ is an $n$-BCP torus for $X$ for some $n$.
\end{definition}

\begin{remark}
	Arguing by induction, one sees that if $T \subset \Res_{R/R_{0}} J_{X}$ is a BCP torus, then the intersection 
	\[
	T \cap j( \Res_{R/R_{0}} X ) \subset \Res_{R/R_{0}} J_{X}\,.
	\]
	has positive dimension. More precisely, if $(X,T)/R$ is a $0$-BCP torus,
	\[
	\dim T \cap j( \Res_{R/R_{0}} X ) = d\,.
	\] 
	If $(X,T)/R$ is the BC-/P-successor of the pair $(Y,T')/R'$, then
	\[
	\dim T \cap j( \Res_{R/R_{0}} X ) = \dim T' \cap j( \Res_{R'/R_{0}} Y )\,.
	\]
\end{remark}

\begin{example}\label{ex:bcp-torus}
	Let $\bbQ \subset K' \subset K$ be number fields with rings of $S$-integers $R_0, R'$, and $R$ respectively. Suppose $K$ has no other subfields. Table~\ref{table:BCP-tori} illustrates the contribution of `iterating' the BCP-construction in the definition in the example of the punctured genus $0$ curve $\mathbb P^1_{R} \smallsetminus \{0,1,2,\infty\}$. We briefly explain some of the entries. 
	
	The only $0$-BCP torus is the full restriction of scalars of the generalized Jacobian, which gives three factors of $\Res_{R/R_{0}}\mathbb G_{m,R}$.
	
	Any $1$-BCP torus arising as a $P$-successor of an $m$-times punctured $\bbP^1_{R}$ would be isogenous to the product of $m-1$ copies of $\Res_{R/R_{0}}\mathbb G_{m,R}$ (the $0$-BCP torus of the image curve) and $4-m$ copies of $\Res_{R/R_{0}}\mathbb G_{m,R}$ coming from the Prym variety. So, (up to isogeny) any $0$-BCP torus arising as a $P$-successor is already a $0$-BCP torus. The $1$-BCP tori arising as BC-successors consist of three factors of $\Res_{R'/R_{0}}\mathbb G_{m,R'}$ or three factors of $\mathbb G_{m,R_0}$. Similarly, for any $\mathbb P^1_{R}$ where the punctures descend to degree $1$ punctures over $R_{0}$, the $1$-BCP tori would be sums of copies of a single irreducible torus.
	
	A BC-successor of a BC-successor can be viewed as a BC-successor in a single step, so to understand the $2$-BCP tori for $\mathbb P^1_{R} \smallsetminus \{0,1,2,\infty\}$, it suffices to understand the $P$-successors of $1$-BCP tori for covered curves. In our case, the covering maps forget or combine individual punctures. As in the previous paragraph, a $P$-successor of an $m$-times punctured $\bbP^1_{R}$ would be isogenous to the product of a $1$-BCP torus and $4-m$ copies of $\Res_{R/R_{0}}\mathbb G_{m,R}$ coming from the Prym variety. So, those $2$-BCP tori which are not also $1$-BCP tori are those products of three total copies of either $\Res_{R/R_{0}}\mathbb G_{m,R}$ and  $\Res_{R'/R_{0}}\mathbb G_{m,R'}$ or $\Res_{R/R_{0}}\mathbb G_{m,R}$ and  $\mathbb G_{m,R_0}$, where not all isogeny factors are the same.
	
	Continuing in this manner, the $3$-BCP tori are (up to isogeny) the product of three total copies of $\Res_{R'/R_{0}}\mathbb G_{m,R'}$ and  $\mathbb G_{m,R_0}$, where not all isogeny factors are the same. The $4$-BCP tori are (up to isogeny) the product $\Res_{R/R_{0}}\mathbb G_{m,R} \times \Res_{R'/R_{0}}\mathbb G_{m,R'} \times \mathbb G_{m,R_0}$. We note that in this example, it is not possible to get three distinct isogeny factors until we have used both the base change and Prym variety constructions twice, so  iterating the construction does provide new examples of BCP-tori. 
\end{example}
	
	\begin{table} 
		\begin{tabular}{|l|c|c|} \hline 
			$n$ & List of $n$-BCP tori for $\mathbb P^1_{R'} \smallsetminus \{0,1,2,\infty\}$ & Explanation \\  \hline 
			$0$ & $\Res_{R/R_{0}}\mathbb G_{m,R} \times \Res_{R/R_{0}}\mathbb G_{m,R} \times \Res_{R/R_{0}}\mathbb G_{m,R}$ & $0$-BCP \\ \hline
			$1$ & $\Res_{R'/R_{0}}\mathbb G_{m,R'} \times \Res_{R'/R_{0}}\mathbb G_{m,R'} \times \Res_{R'/R_{0}}\mathbb G_{m,R'}$ & BC-successor from $\mathbb P^1_{R'} \smallsetminus \{0,1,2,\infty\}$ \\
			    & $\mathbb G_{m,R_0} \times \mathbb G_{m,R_0} \times \mathbb G_{m,R_0}$ & BC-successor from $\mathbb P^1_{R_0} \smallsetminus \{0,1,2,\infty\}$ \\ \hline
			$2$ & $\Res_{R/R_{0}}\mathbb G_{m,R} \times \Res_{R/R_{0}}\mathbb G_{m,R} \times \Res_{R'/R_{0}}\mathbb G_{m,R'}$ & P-successor from $\mathbb P^1_{R} \smallsetminus \{0,\infty\}$ \\
			& $\Res_{R/R_{0}}\mathbb G_{m,R} \times \Res_{R'/R_{0}}\mathbb G_{m,R'} \times \Res_{R'/R_{0}}\mathbb G_{m,R'}$ & P-successor from $\mathbb P^1_{R} \smallsetminus \{0,1,\infty\}$ \\			
			& $\Res_{R/R_{0}}\mathbb G_{m,R} \times \Res_{R/R_{0}}\mathbb G_{m,R} \times \mathbb G_{m,R_0}$ & P-successor from $\mathbb P^1_{R} \smallsetminus \{0,\infty\}$ \\
			& $\Res_{R/R_{0}}\mathbb G_{m,R} \times \mathbb G_{m,R_0} \times \mathbb G_{m,R_0}$ & P-successor from $\mathbb P^1_{R'} \smallsetminus \{0,1,\infty\}$ \\ \hline
			$3$ & $\Res_{R'/R_{0}}\mathbb G_{m,R'} \times \Res_{R'/R_{0}}\mathbb G_{m,R'} \times \mathbb G_{m,R_0}$ & BC-successor from $\mathbb P^1_{R'} \smallsetminus \{0,1,2,\infty\}$ \\
			 & $\Res_{R'/R_{0}}\mathbb G_{m,R'} \times \mathbb G_{m,R_0} \times \mathbb G_{m,R_0}$ & BC-successor from $\mathbb P^1_{R'} \smallsetminus \{0,1,2,\infty\}$ \\ \hline
			$4$ 			& $\Res_{R/R_{0}}\mathbb G_{m,R} \times \Res_{R'/R_{0}}\mathbb G_{m,R'} \times \mathbb G_{m,R_0}$ & P-successor from $\mathbb P^1_{R} \smallsetminus \{0,1,\infty\}$ \\	\hline 
		\end{tabular}
	\caption{Let $\bbQ \subset K' \subset K$ be number fields with rings of $S$-integers $R_0, R'$, and $R$ respectively. Suppose $K$ has no other subfields. This table lists the BCP tori (up to isogeny) for $\mathbb P^1_{R}\smallsetminus\{0,1,2,\infty\}$. The BCP-tori are grouped by the least $n$ for which they are an $n$-BCP torus. The final column explains one way the torus can be constructed as an $n$-BCP torus from an $(n-1)$-BCP torus for another curve, either by base change or by a map that forgets a puncture, although isogenous tori may arise as BCP tori in other ways.}
	\label{table:BCP-tori}
	\end{table}	

\begin{definition}
	Let $C/K$ be a genus $0$ curve, let $\fkm$ be a reduced effective divisor on $C$ and let $X$ be a regular model of $(C,\fkm)$. A  \emph{BCP obstruction} to RoS Chabauty for $X$ is a BCP torus $T$ for $X$ such that 
	\[
	\dim( T \cap j( \Res_{R/R_{0}} X ) ) > \dim T - \dim \overline{T(R_0)}\,.
	\]
\end{definition}

In other words, a BCP obstruction is a subgroup obstruction which is also a BCP torus. Given a BCP obstruction $T$ to RoS Chabauty for $X$, an expected dimension calculation suggests that $j( \Res_{R/R_{0}} X)(\bbZ_{p})\cap \overline{T(R_{0})}$, and therefore $X(\cO_{K} \otimes \bbZ_{p})_{S,1}$, is infinite. 

\subsection{Statement of Main Results}

Our main results show that there are no BCP obstructions to RoS Chabauty for a large collection of subcovers of $\bbP^1_{R}\smallsetminus\{0,1,\infty\}$. For most of the curves we consider, we prove the significantly stronger result that there are no subgroup obstructions to RoS Chabauty. 

Since we expect subgroup obstructions which are not explained by BCP obstructions to be rare (if they exist), we believe these results provide strong evidence that for sufficiently large $q$, the sets $X(\cO_{K} \otimes \bbZ_{p})_{S,1}$ cut out by RoS Chabauty are finite for all of the subcovers $X$ that we consider.

\noBCPobstruction

\nosubgroupobstruction

\begin{remark}
	Recall that $R = \cO_{K,S}$. For any prime $q$ and $\alpha \in R^{\times}$, let $\fkm_{\alpha,q}' = \{x \in \overline{K}: x^q - \alpha = 0\} \cup \{0,\infty\}$, let $\Gamma_{\alpha,q}'$ be the closure of $\supp(\fkm_{\alpha,q}')$ in $\bbP^1_{R}$, and let $X_{\alpha,q}' = \bbP^1_{R} \smallsetminus \Gamma_{\alpha,q}'$.
	
	There is a natural inclusion map $X_{\alpha,q}' \to X_{\alpha,q}$, so to compute $X_{\alpha,q}'(R)$, it suffices to compute $X_{\alpha,q}(R)$. Moreover, elements $X_{\alpha,q}'(R)$ correspond to elements $x \in (\mathbb P^{1}_{R} \smallsetminus \{0,1,\infty\})(R)$ such that $x/\alpha$ is a $q$th power in $R^{\times}$. Since $R^{\times}$ is a finitely-generated abelian group, this means that if one can compute $X_{\alpha,q}'(R)$ as $\alpha$ ranges over a (finite) set of coset representative of $R^{\times}/(R^\times)^q$, then one can compute $(\mathbb P^{1}_{R} \smallsetminus \{0,1,\infty\})(R)$. 
	
	In particular, Theorems~\ref{thm:main-theorem}~and~\ref{thm:main-theorem-2} suggest that it should be possible to compute the set 
	\[
	(\mathbb P^{1}_{R} \smallsetminus \{0,1,\infty\})(R) = \{(x,y)\in R^\times \times R^{\times}: x + y = 1\}
	\]
	of solutions to the $S$-unit equation in $K$ using RoS Chabauty applied to the curves $X_{\alpha,q}$.
\end{remark}

\begin{remark} \label{rem:CM-field-RoS-Obstruction}
	The assumption in Theorem~\ref{thm:main-theorem} that $K$ does not contain a CM-field is necessary under the generalized Leopoldt Conjecture~\ref{conj:general-leopoldt} on $p$-adic linear independence of logarithms of algebraic numbers. In this remark, we exhibit a BCP obstruction to RoS Chabauty for $X_{1,q}$ (defined as in Theorem~\ref{thm:main-theorem})
	under generalized Leopoldt for a single totally real number field.
	
	Suppose $K$ is a CM field and that $q$ is large enough so that $\bbQ(\zeta_{q})$ and $K$ are linearly disjoint. Let $K'$ be the maximal totally real subfield of $K$  with ring of $S$-integers $R'$. Then, there is some $\beta \in R'$ such that $K = K'(\beta)$ and $\beta^2$ is negative in all real embeddings $K' \hookrightarrow \bbR$. Let $\overline{\beta}$ be the Galois conjugate of $\beta$ under the $\Gal(K/K')$ action. Then, $\overline{\beta}$ is the complex conjugate of $\beta$ for every embedding $K \hookrightarrow \bbC$. Define the fractional linear transformation:
	\begin{align*}
	f:  \bbP^1 & \to \bbP^1\,, \\
	x  & \mapsto \frac{\beta x - \overline{\beta}}{x - 1}\,.
	\end{align*}
	Set $\fkm = f(\{ x\in \overline{K}\smallsetminus\{1\} : x^q - 1 = 0\})$. The divisor $\fkm$ is stable under the action of $\Gal(K(\zeta_{q})/K')$ so we may define $Y \colonequals \bbP^1_{R'}\smallsetminus \overline{\supp(\fkm)}$. By construction, $X_{1,q} \cong Y_{R}$.
	
	Set $T = \Res_{R'/R_{0}} J_{Y}$. Assuming generalized Leopoldt, $\dim \overline{T(R_{0})} = \min( \dim T, \rank T(R_{0}))$. If this holds, we claim that the $1$-BCP torus $T$ is a BCP obstruction to RoS Chabauty for $X_{1,q}$.
	
	In every complex embedding, the map $f$ takes the unit circle to the real axis, so the field $K'(\Gamma)$ is totally real. Applying \eqref{eqn:jacobian-punctured-P1-rank} gives 
	\[
	\rank(T(R_{0})) = \rank J_{Y}(R') \geq [K':\bbQ](q-2) = \dim \Res_{R'/R_{0}} J_{Y}\,,
	\]
	with equality if and only if every prime in $S'$ splits completely in $K'(\Gamma)$.	Generalized Leopoldt implies $\dim \overline{T(R_{0})} = \dim T$. Since  
	\[
	\dim(T \cap j(\Res_{R/R_{0}} X_{1,q})) = \dim \Res_{R'/R_{0}} Y = [K':\bbQ] > 0\,,
	\]
	we see that $T$ is indeed a BCP obstruction to RoS Chabauty for $X_{1,q}$.
	
	Of course, if $K$ is not CM, but contains a CM field, we may apply the same argument to the CM subfield to construct a base change obstruction to RoS Chabauty for $X_{1,q}$.
	
	In contrast, we note that Theorem~\ref{thm:main-theorem-2} requires no assumptions on $K$.
\end{remark}

\section{Classical Chabauty and genus $0$ descent.} \label{sec:classical-chab+descent}

Suppose that $C/K$ is a genus $0$ curve, $\fkm$ is a reduced effective divisor on $C$ and that $X = \overline{X} \smallsetminus \Gamma$ is a regular model for $(C,\fkm)$ over $\OKS = \cO_{K,S}$. Let $J$ be the generalized Jacobian of $X$. Under the generalized Leopoldt conjecture~\ref{conj:general-leopoldt}, the classical Chabauty set $X(\cO_{K_{\fkp}})_{S,1}$ containing $X(\OKS)$ is finite whenever there is some torus $T \subset J$ such that 
\begin{align}\label{eqn:classical-chab-genus-0}
\rank T(\OKS) < \dim T\,.
\end{align}

We start by proving a bound on the ranks of tori over the rings of integers of number fields. Recall that a torus $T$ over $K$ is \emph{anistropic} if $\hom_{K}(T, \mathbb G_{m}) = \{1\}$. Equivalently, the isogeny decomposition of $T$ as a product of irreducible tori over $K$ does not contain any copies of $\mathbb G_{m}$. Similarly, we say that a torus over $\cO_{K}$ or $\cO_{K,S}$ is anisotropic if its generic fiber is anisotropic. The following lemma is likely known to experts but we could not find a reference:

\begin{lemma} \label{lem:rank-anisotropic}
	Suppose that $T/\mathcal O_{K}$ is an anisotropic torus. Then,
	\begin{align*}
	r_{2}(K) \dim T \leq \rank T(O_{K}) \leq (r_{1}(K) + r_{2}(K)) \dim T\,.
	\end{align*}	
\end{lemma}

\begin{proof}
	Let $L/K$ be a Galois extension that splits $T_{K}$ and set $G = \Gal(L/K)$. Then, $G$ acts on the character lattice $X(T) \mathrm{Hom}_{L}(T, \mathbb G_{m})$. Let $\chi_{T}$ be the associated $\bbC$-valued representation of $G$. The group $G$ also acts on $\cO_{L}^{\times}$. Let $\chi_{L}$ be the associated $\bbC$-valued representation.
	By \cite[Corollary~6.9]{eisentrager-03}, $\rank T(\cO_{K}) = \langle \chi_{T}, \chi_{L}\rangle$. The group $G$ also acts on the infinite places of $L$. Let $\psi$ be the corresponding representation of $G$ acting on $\bbC^{r_{1}(L) + r_{2}(L)}$, viewed as formal linear combinations of infinite place of $L$. We have $\psi \cong \chi_{L} \oplus \mathbf{1}$. Moreover, since $T$ is anisotropic, $\langle \chi_{T}, \mathbf{1}\rangle = 0$. So, $\rank T(\cO_{K}) = \langle \chi_{T}, \psi \rangle$.
	
	We now consider the structure of $\psi$. The action of $G$ permutes the set of places above any particular infinite place of $K$. Moreover, if there are $\# G$ places above a certain place of $K$ then the action on the set of places of $L$ above this place of $K$ is by the right regular representation $\mathrm{Ind}^{G}_{\{1\}} \mathbf{1}$. This occurs if the place of $K$ is complex or if the place of $K$ is real and the places above it are also real.	
	 Alternately, if there are $\frac{1}{2}\# G$ places above a certain place of $K$, then the stabilizer of some place is a subgroup $H\subset G$ of order $2$ and the representation on this set of places is by $\mathrm{Ind}^{G}_{H} \mathbf{1}$. This case occurs if the place of $K$ is real and the places above it are complex. Let $m_{\rho}(\psi)$ denote the multiplicity of the irreducible representation $\rho$ in $\psi$ and let $m_{\rho}(\chi_{T})$ denote the multiplicity of $\rho$ in $\chi_{T}$. Since every irreducible representation $\rho$ of $G$ occurs with multiplicity $\dim \rho$ in the right regular representation and with multiplicity at most $\dim \rho$ in $\mathrm{Ind}^{G}_{H} \mathbf{1}$, we conclude
	\begin{align}\label{eqn:irrep-multiplicity}
	r_{2}(K) \dim \rho \leq m_{\rho}(\psi) \leq (r_{1}(K) + r_{2}(K)) \dim \rho\,.
	\end{align}
	Now, since $\rank T(\cO_{K}) = \langle \chi_{T}, \psi \rangle$,
	\[
	\rank T(O_{K}) = \sum_{\mathrm{irreps. } \,  \rho \, \mathrm{ of }\,  G} m_{\rho}(\chi_{T}) m_{\rho}(\psi)\,.
	\]
	Applying \eqref{eqn:irrep-multiplicity} to each $m_{\rho}(\psi)$ and using the fact that \[\dim T = \sum_{\mathrm{irreps. } \,  \rho \, \mathrm{ of }\,  G} m_{\rho}(\chi_{T}) \dim \rho\,,\] we arrive at the desired result.
\end{proof}

To apply Lemma~\ref{lem:rank-anisotropic} to study Chabauty's method, we will consider the case where $T$ is a subtorus of the generalized Jacobian $J$.

\begin{theorem}\label{thm:classical-Chabauty}
Let $K$ be a number field which is not totally real. Let $S$ be any finite set of finite places of $K$, let $\mathfrak p$ be a finite place of $K$ which is not in $S$, and let $\fkm$ be a reduced effective divisor on $\mathbb P^1_{K}$. Let $X/\mathcal O_{K,S}$ be a regular model of $(\mathbb P^1_{K}, \fkm)$. 

Suppose that $[K: \mathbb Q] \geq 3$, that the generalized Leopoldt Conjecture~\ref{conj:general-leopoldt} holds over $K$. Then, $X(\cO_{K_{\fkp}})_{S,1} = X(\mathcal O_{K_{\mathfrak p}})$. In particular, the set of points  $X(\cO_{K_{\fkp}})_{S,1}$ cut out by the classical Chabauty's method is either empty or infinite.
\end{theorem}

\begin{proof}
Any such $X$ spreads out to a regular model of $(\mathbb P^1_{K},\fkm)$ over $\mathcal O_{K}$, so it suffices to consider the case $S = \emptyset$. 
	
Let $J$ be the generalized Jacobian of $X$. Since $[K: \mathbb Q] \geq 3$, we have 
\begin{align} \label{eqn:rank-Gm}
\rank \mathbb G_{m}(\mathcal O_{K,S}) = r_{1}(K) + r_2(K) + \#S - 1 \geq 1 = \dim \mathbb G_{m}\,.
\end{align}
 Since $K$ is not totally real, combining Lemma~\ref{lem:rank-anisotropic} with equation \eqref{eqn:rank-Gm} implies that for any irreducible torus $T \subset J$, we have
 \(
 \rank T(\mathcal O_{K}) \geq \dim T\,.
 \)
 In particular, the generalized Leopoldt Conjecture~\ref{conj:general-leopoldt} implies that $\Span_{K_{\fkp}} \log T(\cO_{K}) = K_{\fkp}^{\dim T}$. Altogether, we see that $\Span_{K_{\fkp}} \log J(\cO_{K}) = K_{\fkp}^{\dim J}$, and so $X(\cO_{K_{\fkp}})_{S,1} = X(\mathcal O_{K_{\mathfrak p}})\,.$ To conclude, we note that $ X(\mathcal O_{K_{\mathfrak p}})$ is either empty or infinite.	
\end{proof}

As an immediate consequence of Theorem~\ref{thm:classical-Chabauty}, we have:

\begin{corollary}\label{cor:Chabauty+descent}
Let $K$ be a number field, let $S$ be any finite set of finite places of $K$, and let $\mathfrak p$ be a finite place of $K$ which is not in $S$. Suppose that $[K: \bbQ] \geq 3$, that $K$ is not totally real, and that the generalized Leopoldt Conjecture~\ref{conj:general-leopoldt} holds over $K$.

Then, for any finite set $\mathcal D$ of genus zero covers $f_{X}: X \to \bbP^1_{\OKS} \smallsetminus \{0,1,\infty\}$ satisfying
\[
(\bbP^1_{\OKS} \smallsetminus \{0,1,\infty\})(\OKS) \subset \bigcup_{X \in \mathcal D} f_{X}(X(\OKS))\,,
\]
 the set of points 
\[
\bigcup_{X \in \mathcal D} f_X(X(\cO_{K_{\fkp}})_{S,1}) \subset (\bbP^1_{\OKS} \smallsetminus \{0,1,\infty\})(\cO_{K_{\fkp}})
\]
 cut out by the classical Chabauty's method plus descent is either empty or infinite.
\end{corollary}

\begin{remark}	
	In contrast to Corollary~\ref{cor:Chabauty+descent}, when $K = \bbQ$ or is imaginary quadratic, \cite{poonen-19} shows that when $\mathcal D$ is the collection of $(\bbZ/q\bbZ)$-covers of $\bbP^1_{R}\smallsetminus \{0,1,\infty\}$ which extend to be totally ramified above $0$ and $\infty$ and \'etale elsewhere, the set 
	\[
	\bigcup_{X \in \mathcal D} f_{X}\left(X(\cO_{K_{\fkp}})_{S,1}\right) \subset (\bbP^1_{\OKS} \smallsetminus \{0,1,\infty\})(\cO_{K_{\fkp}})
	\]
	is finite. 
\end{remark}

\begin{remark}
	If $K$ is totally real and $[K:\mathbb Q] + \# S \geq 4,$ one can check that for any regular model $X$ of $(\mathbb P^1_{K}, \fkm)$, with generalized Jacobian $J$, we have $\rank  J(\mathcal O_{K,S}) \geq \dim J$. Naively, one might expect this to imply that $X(\cO_{K_{\fkp}})_{S,1}$ is infinite when it it nonempty, but this need not be the case. For instance, if $S = \emptyset$ and one of the Galois orbits of the divisor $\fkm$ defines a CM field containing $K$, then $J$ will have a nontrivial subtorus $T$ with $\rank T(\mathcal O_{K}) = 0$.
	
	To give a concrete example, if $\fkm = \{\pm i\}$, then $ J = \Res_{\cO_{K[i]}/\cO_{K}} \mathbb G_{m}$ has a subtorus which is isogenous to the quotient $T =  J/\mathbb G_{m}$. Since 
	\[
	\rank  J(\cO_{K}) = \rank \cO_{K[i]}^{\times} = \rank \cO_{K}^{\times} = \rank \mathbb G_{m}(\mathcal O_{K})\,,
	\] we see that $\rank T(\cO_{K}) = 0$. In particular, $X(\cO_{K_{\fkp}})_{S,1}$ is finite even though the Chabauty inequality $\rank J(\mathcal O_{K,S}) \geq \dim J$ is not satisfied for $X$. 
	
	More generally, let $\fkm_{1,q}' = \{0,\infty\} \cup \{x: x^{q} - 1 = 0\}$ and let $X_{1,q}'$ be a regular model of $(\bbP^1_{K}, \fkm_{1,q}')$ over $\cO_{K}$. The generalized Jacobian of $X_{1,q}'$ has a subtorus $T$ of dimension $(q-1)/2$ with $\rank T(\cO_{K}) = 0$. Arguing as in the proof of Lemma~\ref{lem:rank-of-subtori-2}, one can show for that $\rank T(\cO_{K,S}) < (q-1)/2$ and so $X_{1,q}'(\cO_{K_{\fkp}})_{S,1}$ is finite if $q$ is large enough. 
\end{remark}

\section{No BCP obstructions to RoS Chabauty for $\mathbb P^1 \smallsetminus \{x: x^q - 1 = 0, x \neq 1\}$}\label{sec:RoS-chab+descent}
	
Let $\fkm_{1,q} = \{x: x^q - 1 = 0, x \neq 1\}$ and let $\Gamma_{1,q}$ be the closure of the support of $\fkm_{1,q}$ in $\bbP^1_{\OKS}$. In this section, we show that when $K$ does not contain a CM subfield, and for a sufficiently large prime $q$, there are no BCP obstructions to RoS Chabauty for $\mathbb P^1_R \smallsetminus \Gamma_{1,q}$. Our main goal is to prove Theorem~\ref{thm:main-theorem}. 

This case is particularly difficult to control for the purposes of RoS Chabauty for two closely related reasons. First, $\Res_{R[\zeta_{q}]/R_{0}} \bbG_{m}/\Res_{R/R_{0}} \bbG_{m}$ has many subtori, including many subtori with large rank if $K$ has a large totally real subfield. Understanding how translates of these subtori intersect other subvarieties seems to be a subtle geometric problem. Second, since the Galois group $\Gal(K[\zeta_{q}]/K)$ is abelian, it imposes less structure than the nonabelian Galois group $\Gal(K[\zeta_{q}, \sqrt[q]{\alpha}]/K)$ will impose in Section~\ref{sec:RoS-chab+descent-2}. 

Before we begin a series of technical lemmas, we sketch out the proof. Given a BCP torus $T$, Lemma~\ref{lem:rank-deficiency} allows us to bound $\dim T - \dim \overline{T(R_0)}$ from below in terms of ranks and dimensions of the generalized Jacobians of certain curves. Theorem~\ref{thm:rank-bounds} shows that if $X/\OKSp$ is any curve which becomes isomorphic to $X_{\alpha,q}$ after base change to $R$, then  the rank of the $\OKSp$ points of the generalized Jacobian $J_X$ of $X$ is `small'. More precisely, for large $q$, there is a constant $c > 0$ such that $\rank J_X(\OKSp) + cq < \dim J_X$. Lemma~\ref{lem:covers-rank} shows that if the punctures of $X$ form a single Galois orbit,  $(X,T)$ is a $P$-successor of some $(Y,T')$, and $\rank J_{X}(R)$ is `small' in the sense above, then for large $q$ 
the difference $\left( \rank J_X(\OKSp) - \rank J_{Y}(\OKSp) \right) $ is `small' compared to $[K':\bbQ](\dim J_{X} - \dim J_{Y})$.
Combining these lemmas, any BCP torus $T$ for $X_{1,q}$ satisfies $\dim T - \dim \overline{T(R_0)} > \dim \Res_{R/R_{0}} X_{1,q}$, so there is no BCP obstruction to RoS Chabauty for $X_{1,q}$.

\begin{lemma} \label{lem:rank-deficiency}
	Suppose that $T$ is an $n$-BCP torus for $X/R$. Set $T_n = T, R_n' = R,$ and $X_n = X$. For each $i \in \{0, \dots, n-1\}$ there is a ring of $S$-integers $R_{i}'$ with fraction field $K_{i}$, a curve $X_{i}/R_{i}'$ with generalized Jacobian $J_{i}$, and an $i$-BCP torus $T_{i}$ such that $T_0$ is a $0$-BCP torus for $X_0$ and $(X_{i+1}, T_{i+1})$ is either a BC-successor or a P-successor of $(X_{i}, T_{i})$. 
	
	Set
	\begin{align*}
	\delta_{i} = 
	\begin{cases}
	[K_{0}:\bbQ] \cdot \dim J_{0} - \rank J_{0}(R_{0}') & \text{ if } i = 0\,,\\[1ex]
	0 & \text{ if } (X_{i}, T_{i}) \text{ is a BC-successor of } (X_{i-1}, T_{i-1})\,,\\[1ex]
	\!\begin{aligned}[t]
	&[K_{i}:\bbQ] \cdot (\dim J_{i} - \dim J_{i-1}) \\
	&\quad 	- (\rank J_{i}(R_{i}') - \rank J_{i-1}(R_{i-1}') )
	\end{aligned}
	& \text{ if } (X_{i}, T_{i}) \text{ is a P-successor of } (X_{i-1}, T_{i-1})\,.
	\end{cases}
	\end{align*}
	Then,
	\[
	\dim T - \dim \overline{T(R_{0})} \geq \sum_{i=0}^{n} \max(0, \delta_{i})\,.
	\]
\end{lemma}

\begin{proof}
	The proof is by induction on $n$.
	
	For $n = 0$, we have $T_0 = \Res_{R_{0}'/R_{0}} J_0$, so the claim is equivalent to $\rank T_{0}(R_{0}') \geq \dim \overline{T_{0}(R_{0}')}$, which holds because the $\log$ map respects the $p$-adic topology.
	
	Suppose the claim holds for $(X_{n-1}, T_{n-1})$. If $(X_{n},T_{n})$ is the BC-successor of $(X_{n-1}, T_{n-1})$ then $T_{n} \cong T_{n-1}$ and $\delta_{n} = 0$, so the claim is immediate. 
	
	Otherwise, $(X_{n},T_{n})$ is the P-successor of  $(X_{n-1}, T_{n-1})$, so there is a map $f: X_{n} \to X_{n-1}$ and $R_{n}' = R_{n-1}'$. Let $O$ be the identity in $J_{n-1}$ and set $P = f_{*}^{-1}(O)$ for $f_{*}$ as in Definition~\ref{def:p-successor}. Up to isogeny, $T_{n} \sim  P \times T_{n-1}$ and $\Res_{R_{n}'/R_{0}} J_{n} \sim P \times  \Res_{R_{n-1}'/R_{0}} J_{n}$. So,
	\begin{align*}
	\dim T_{n} -  \dim \overline{ T_{n}(R_{0})} & = \left( \dim T_{n-1} - \dim \overline{ T_{n-1}(R_{0})}\right) + \left(\dim P - \dim \overline{P(R_{0})}\right) \\
	& \geq \sum_{i=0}^{n-1} \max(0, \delta_{i}) + \max(0, \dim P - \rank P(R_{0})) \\
	& = \sum_{i=0}^{n-1} \max(0, \delta_{i}) + \max(0, \delta_{n})\,.
	\end{align*}
\end{proof}

We now begin a series of lemmas aimed at bounding the ranks of the generalized Jacobians of certain genus $0$ curves, culminating in Theorem~\ref{thm:rank-bounds}. 

Note that if $\bbP^1_{\OKS}\smallsetminus \Gamma_1$ and $\bbP^1_{\OKS}\smallsetminus \Gamma_2$ are isomorphic, then there is some fractional linear transformation $\phi: x \mapsto \frac{ax + b}{cx + d}$ with $a,b,c,d \in K$ such that $\phi(\Gamma_1(\overline{K})) = \Gamma_2(\overline{K})$. To bound the rank of the generalized Jacobians of $\bbP^1_{\OKS}\smallsetminus \Gamma_1$ we will need to understand the number of real and complex embeddings of the fields generated by the Galois orbits in $\Gamma_1$, as well as the splitting of the primes in $S$ in these fields.

\begin{lemma} \label{lem:helper-lemma}
	Suppose that $K' \subset K$ are number fields. Let $X$ be a genus $0$ curve defined over $K'$ equipped with an isomorphism $\phi: \bbP^1_{K} \to X_{K}$. Fix $P \in \bbP^1_{K}(\overline{\bbQ})$. Suppose also that $[K'(\phi(P)): K'] = [K(P): K]$.  Given $r \in \bbR_{>0}$, let
	\[
	U_r \colonequals \{x \in \bbC: |x| = r\}
	\]
	be the circle of radius $r$ in $\bbC \subset \bbP^1(\bbC)$.
	
	Let $I$ be a set of embeddings $\iota: K \to \bbC$. For each $\iota \in I$, fix some $r_{\iota} \in \bbR_{>0}$. Suppose that for all embeddings $\iota': K(P) \hookrightarrow \bbC$ extending $\iota$, we have $\iota'(P) \in U_{r_{\iota}}$. Set
	\[
	K_{\text{bad}}~\colonequals \{x \in \bbP^1(K): \iota(x) \in U_{r_{\iota}} \text{ for all } \iota \in I\}\,.
	\]
	If $K_{\text{bad}}$ is finite and either
	\begin{enumerate}
		\item $K_{\text{bad}} \neq \emptyset $\ or
		\item $X \cong \bbP^1_{K'} $
	\end{enumerate}
	then for some $\iota \in I$ there are at most $2$ real embeddings $K'(\phi_{\iota}(P)) \hookrightarrow \bbR$ extending $\iota|_{K'}$.
\end{lemma}

\begin{proof}[Proof of Lemma~\ref{lem:helper-lemma}]
	
	Observe first that $[K'(\phi(P)): K'] = [K(P): K]$ implies that embeddings $K'(\phi(P)) \hookrightarrow \bbC$ extending $\iota|_{K'}$ correspond exactly to embeddings $\iota': K(P) \hookrightarrow \bbC$ extending $\iota$.
	
	We now prove the contrapositive. Suppose that for each $\iota \in I$ there are at least three $\iota': K(P) \to \bbC$ so that $\iota'(K'(\phi(P))) \subset \bbR$. 
	
	Then, each $\iota|_{K'}: K' \hookrightarrow \bbC$ is a real embedding. Given $\iota: K \hookrightarrow \bbC$, let $X_{\iota}$ be the base change of $X$ from $K'$ to $\bbR$ along $\iota$. Let $\phi_{\iota}$ denote the induced map $\bbP^1_{\bbC} \to X_{\iota, \bbC}$. 
	
	If $X \cong \bbP^1_{K'}$, let $\psi: X \to \bbP^1_{K'}$ be the identity map. The induced isomorphisms $\psi_{\iota}: X_{\iota} \to \bbP^1_{\bbR}$ and $\psi_{\iota}: X_{\iota,\bbC} \to \bbP^1_{\bbC}$ are again the identity.
	
	Otherwise, $X$ can be written as a conic in $\bbP^2_{K'}$. Choose some $Q \in K_{\text{bad}}$ and let $\psi: X_{K} \to \bbP^1_{K}$ be projection from $\phi(Q)$ to a line. Since $\iota(\phi(Q)) \in \bbR$ for all $\iota \in I$, there are again induced isomorphisms $\psi_{\iota}: X_{\iota} \to \bbP^1_{\bbR}$ and $\psi_{\iota}: X_{\iota,\bbC} \to \bbP^1_{\bbC}$ given by projection from $\iota(\phi(Q))$. 
	
	In either case, $\psi_{\iota, \bbC}$ maps $X_{\iota}(\bbR)$ isomorphically to the real axis in $\bbP^1(\bbC)$. 
	
	Now, by assumption, for any $\iota \in I$, we can choose distinct $\iota_{1}, \iota_{2}, \iota_{3}$ extending $\iota$ to $K(P)$ such that $\phi_{\iota}(\iota_{j}(P)) \in X_{\iota}(\bbR)$ for $j \in \{1,2,3\}$. Then, $(\psi_{\iota} \circ \phi_{\iota})(\iota_{j}(P))$ lies on the real axis in $\bbP^1(\bbC)$ for $j \in \{1, 2, 3\}$. Since any automorphism of $\bbP^1(\bbC)$ which maps three points on a circle to a (real) line induces a bijective map between the circle and the line, we see that the composition $\psi_{\iota} \circ \phi_{\iota}$ maps $U_{r_{\iota}}$ bijectively to the real axis in $\bbP^1(\bbC)$. Hence, $\phi_{\iota}$ maps $U_{r_{\iota}}$ bijectively to $X_{\iota}(\bbR)$.
	
	Now, consider the map $\psi \circ \phi : \bbP^1_{K} \to \bbP^1_{K}$. Since, $\bbP^1$ is definable over $\bbQ$, it makes sense to consider $\bbP^1(\bbQ)$ as a set inside $\bbP^1(K)$. Moreover, for each $\iota \in I$, the set $\bbP^1(\bbQ)$ is contained in the real axis of $\bbP^1(\bbC)$. In particular, for all $x \in \bbP^1(\bbQ)$ and all $\iota \in I$, we have 
	\[
	\iota((\psi \circ \phi)^{-1}(x)) \in U_{r_{\iota}}\,.
	\]
	Hence, 
	\[
	(\psi \circ \phi)^{-1}(\bbP^1(\bbQ)) \subset K_{\text{bad}}\,.
	\]
	so $K_{\text{bad}}$ is infinite. 
\end{proof}

Recall that a CM field is a totally complex field which is a degree $2$ extension of a totally real field. We state and prove a well-known fact characterizing fields containing CM subfields.

\begin{lemma}\label{lem:CM-subfield}
	Let $K$ be a number field and fix $\alpha \in K$. Then, $K$ contains a CM subfield if and only if the set
	\[
	C_{\alpha} \colonequals \{ x \in K : |\iota(x)| = |\iota(\alpha)| \text{ for all } \iota: K \hookrightarrow \bbC \}
	\]
	is infinite, or equivalently, if $\# C_{\alpha} \geq 3$\,.
\end{lemma}

\begin{proof}[Proof of Lemma~\ref{lem:CM-subfield}]
	Dividing by $\alpha$ if necessary, it suffices to prove the claim for $\alpha = 1$. 
	
	If $ \# C_{1} \geq 3$ then there is some $\beta \neq \pm 1$ on the unit circle in every complex embedding. The same is true of $\beta^{-1}$, so $\bbQ(\beta)$ is a totally complex degree two extension of $\bbQ\left(\frac{\beta + \beta^{-1}}{2}\right)$\,.
	
	If $K$ contains a CM subfield, suppose $K'(\beta)$ is a CM subfield containing the totally real subfield $K'$. Let $\sigma \in \Gal(K'(\beta)/K')$ be the nontrivial element corresponding to complex conjugation. Then, $\beta - \sigma(\beta) \neq 0$ lies on the imaginary axis in every embedding $K \hookrightarrow \bbC$. Finally, the fractional linear transformation $x \mapsto \frac{x+1}{x-1}$ maps the imaginary axis to the unit circle in every embedding. Hence, 
	\[
	\left\{ \frac{n\beta - n\sigma(\beta) + 1}{n\beta - n\sigma(\beta) - 1} : n \in \bbQ \right\} \subset C_{1}
	\]
	so $C_{1}$ is infinite.
\end{proof}

\begin{lemma}\label{lem:odd-degree-field}
	Suppose that $[K:\bbQ]$ is odd. Fix an $\iota: K \hookrightarrow \bbC$ and an $r \in \bbR_{>0}$. Then,
	\[
	\# \{ x \in K : |\iota(x)| = r \} \leq 2 \,.
	\]
\end{lemma}

\begin{proof}[Proof of Lemma~\ref{lem:odd-degree-field}]
	Suppose not. Then there are $\alpha, \beta \in K$ with $\iota(\alpha) = \iota(\beta) = r$ and $\alpha \neq \pm \beta$. Set $x = \alpha/\beta$, so $|\iota(x)| = 1$. Then, $\iota(x) \notin \bbR$ but $\iota(x) + \iota(x)^{-1} \in \bbR$. Hence, $[\bbQ(x):\bbQ(x + x^{-1})] = 2$. But then $[K:\bbQ] = [K:\bbQ(x)] \cdot  [\bbQ(x):\bbQ]$ is even.
\end{proof}

\begin{lemma}\label{lem:qth-root-of-unity}
	Suppose that $K' \subset K$ are number fields. Let $q$ be a prime. Let $X$ be a genus $0$ curve defined over $K'$ equipped with an isomorphism $\phi: \bbP^1_{K} \to X_{K}$. 
	Suppose that the set $Z_{q} \colonequals \left\{\zeta_{q}^{j}: j \in \{1, \dots, q-1\} \right\}$ of primitive $q$th roots of unity defines an irreducible degree $q-1$ point of $\bbP^1_{K}$ and that $\{\phi(x): x \in Z_{q}\}$ descends to an irreducible degree $q-1$ point of $X$.  
		Then,
		\begin{enumerate} [label=\alph*)]
			\item If $[K:\bbQ]$ is odd, 
			\[
			r_1(K'(\phi(\zeta_q))) \leq 2 r_1(K') \,.
			\]
			\item If $K$ does not contain a CM subfield, 
			\[
			r_1(K'(\phi(\zeta_q))) \leq [K':\bbQ](q-1)  - (q - 3)\,.
			\]
		\end{enumerate}
\end{lemma}

\begin{proof}[Proof of Lemma~\ref{lem:qth-root-of-unity}]
	The assumption that we have irreducible points implies that
	\begin{align*}
	[K(\zeta_q): K] = [K'(\phi(\zeta_q)): K'] = q-1\,.
	\end{align*}
	
	(a): Suppose $[K: \bbQ]$ is odd. Fix an $\iota: K \hookrightarrow \bbC$, set $I = \{\iota\}$, and set $r_{\iota} = 1$. Define $K_{\text{bad}}$ as in Lemma~\ref{lem:helper-lemma}. By Lemma~\ref{lem:odd-degree-field}, $\# K_{\text{bad}} \leq 2$.
	
	We have $\pm 1 \in K_{\text{bad}}$, so $X \cong \bbP^1_{K'}$. Then, Lemma~\ref{lem:helper-lemma} says there are at most $2$ real embeddings of $K'(\phi(\zeta_{q}))$ extending $\iota|_{K'}$. Of course, if $\iota|_{K'}$ is not real, there are no such embeddings. We conclude
	\begin{align*}
	r_1(K'(\phi(\zeta_q))) & \leq 2 r_1(K') \,.
	\end{align*}
	
	(b): Suppose $K$ does not contain a CM subfield. Let $I$ be the set of all embeddings $\iota: K \hookrightarrow \bbC$ and set $r_{\iota} = 1$ for all $\iota$. Define $K_{\text{bad}}$ as in Lemma~\ref{lem:helper-lemma}. By Lemma~\ref{lem:CM-subfield}, $\# K_{\text{bad}} \leq 2$.
	
	We again have $\pm 1 \in K_{\text{bad}}$. By Lemma~\ref{lem:helper-lemma},
	\begin{align*}
	r_1(K'(\phi(\zeta_q))) & \leq ([K':\bbQ] - 1)(q-1) + 2 = [K':\bbQ](q-1) - (q-3)  \,, \\
	r_1(K'(\phi(\sqrt[q]{\alpha}))) & \leq ([K':\bbQ] - 1)q + 2 = [K':\bbQ]q  - (q-2)\,.
	\end{align*}
\end{proof}

\begin{theorem}\label{thm:rank-bounds}
	Let $K$ be a number field which does not contain a CM subfield and let $S$ be a finite set of finite places. For any $\varepsilon > 0$, there exists $C \in \bbR$ such that for all primes $q > C$ the following holds: 
	
	Let $K'$ be any subfield of $K$. Let $S'$ be the set of places of $K'$ lying under $S$ and let $R' = \cO_{K',S'}$ be the ring of $S'$-integers of $K'$. Choose any $\alpha \in \OKS^{\times}$ which is a $q$th power in $\OKS$ and let $\Gamma$ be the divisor on $\bbP^1_{R}$ given by spreading out the divisor $\{ x \in \overline{R}\smallsetminus R : x^q - \alpha = 0 \}$ on $\bbP^1_{K}$.
	
	Let $X/R'$ be any curve such that there is an isomorphism $\phi : \mathbb P^1_{R} \smallsetminus \Gamma \to X_{R}\,.$ Then,
	\[
	\rank J_{X}(R') \leq \left( [K':\bbQ] - \frac{1}{2} + \varepsilon \right) (q-2)\,.
	\]	
\end{theorem}

\begin{proof}[Proof of Theorem~\ref{thm:rank-bounds}]
	
	It is enough to consider the case $\alpha = 1$.
	
	Note that $\mathbb P^1_{R} \smallsetminus \Gamma \cong \mathbb P^1_{R} \smallsetminus \{\zeta_{q}^{j}: j \in \{1, \dots, q-1\} \}\,.$ By choosing $C$ sufficiently large, we may assume that $\bbQ(\zeta_{q})$ is linearly disjoint from $K$. Then,
	\[
	[K(\zeta_{q}):K] = [K(\phi(\zeta_{q})):K] \leq [K'(\phi(\zeta_{q})):K']\,,
	\]
	so the punctures of $X_{K'}$ form an irreducible degree $q-1$ point on $\overline{X_{K'}}$.
	
	Using Lemma~\ref{lem:rank+dim-genus-0-jac} or equation \eqref{eqn:jacobian-punctured-P1-rank} to bound the rank of $J_{X}(R')$, it will suffice to bound the number of infinite places of $K'(\phi(\zeta_q))$ and the number of finite places of $K'(\phi(\zeta_q))$ above $S'$.
	
	We start with the infinite places. By 1b of Lemma~\ref{lem:qth-root-of-unity}, 
	\[
	r_1(K'(\phi(\zeta_q))) \leq [K':\bbQ](q-1)  - (q - 3)\,.
	\]
	Also, 
	\[
	r_1(K'(\phi(\zeta_q))) + 2 r_{2}(K'(\phi(\zeta_q))) = [K':\bbQ](q-1)\,,
	\]
	so
	\begin{align}\label{eqn:infinite-places-1}
	r_1(K'(\phi(\zeta_q))) + r_{2}(K'(\phi(\zeta_q))) \leq [K':\bbQ](q-1) - \frac{q - 3}{2}\,.
	\end{align}
	
	To address the finite places, we note that the total number of places of $K'(\phi(\zeta_q))$ above $S'$ is at most the total number of places of $K(\zeta_{q})$ above $S$. 
	
	Fix a prime $\fkp \in S$. Let $\kappa_{\fkp}$ be the residue field of $K_{\fkp}$. Let $a_{\fkp} = [\kappa_{\fkp}(\zeta_q):\kappa_{\fkp}]$. 
	Then, $a_{\fkp}$ is the order of $\# \kappa_{\fkp}$ in $(\bbZ/q\bbZ)^\times$, so 
	\[
	a_{\fkp} \geq \left \lceil \frac{\ln (q - 1)}{\ln(\#\kappa_{\fkp})} \right \rceil\,.
	\]
	Also, 
	\[
	\#\{\text{primes } \fkP \text{ of } K(\zeta_{q}) \text{ above } \fkp\} = \frac{q-1}{a_{\fkp}} \leq \frac{q-1}{\left \lceil \frac{\ln (q - 1)}{\ln(\#\kappa_{\fkp})} \right \rceil}\,.
	\]
	The set $S$ is finite and the denominators shrink as $q$ grows, so there is some $C$ (depending only on $\varepsilon, K$, and $S$) such that for $q > C$ we have
	\begin{align}\label{eqn:finite-prime-bound-1}
	\sum_{\fkp \in S} \#\{\text{primes } \fkP \text{ of } K'(\phi(\zeta_{p})) \text{ above } \fkp\} \leq \frac{\varepsilon}{2} (q-1)\,.
	\end{align}
	Combining \eqref{eqn:infinite-places-1} and \eqref{eqn:finite-prime-bound-1} in \eqref{eqn:jacobian-punctured-P1-rank} gives  
	\begin{align*}
	\rank J_{X}(R') \leq \left( [K':\bbQ] - \frac{1}{2} + \frac{\varepsilon}{2} \right) (q-1) + \frac{3}{2}\,.
	\end{align*}
	
	Increasing $C$ if necessary, we may assume $\varepsilon q/2$ is larger than any given constant. We find
	\begin{align*}
	\rank J_{X}(R') \leq \left( [K':\bbQ] - \frac{1}{2} + \varepsilon \right) (q-2)\,.
	\end{align*}
\end{proof}

\begin{lemma}\label{lem:covers-rank}
	Let $C_1, C_2/K$ be genus $0$ curves equipped with reduced effective divisors $\fkm_1, \fkm_2$. Let $X_1 \colonequals \overline{X_{1}} \smallsetminus \Gamma_1$ and $X_2 \colonequals \overline{X_{2}} \smallsetminus \Gamma_2$ be regular models of $(C_1,\fkm_{1})$ and $(C_2, \fkm_{2})$, respectively. Let $J_{1}$ and $J_{2}$ be the generalized Jacobians of $X_{1}$ and $X_{2}$. Let $G = \Gal(\overline{K}/K)$. 
	
	Suppose we are given a morphism $\varphi: X_1 \to X_{2}$ such that the induced morphism $\varphi: C_1 \to C_2$ satisfies $\varphi(\Gamma_1(\overline{K})) = \Gamma_2(\overline{K})$. 
	Suppose also that $\#(G\backslash \Gamma_{2}(\overline{K})) = 1$ and set 
	$\delta = \#  \Gamma_{1}(\overline{K}) /\# \Gamma_{2}(\overline{K})$.	Set 
	\[
	\Delta \colonequals [K:\bbQ] \cdot \# \Gamma_1(\overline{K}) - 
	(\rank J_{1}(\OKS) + \rank \bbG_{m}(\OKS) + 1) \,.
	\]
	
	Then
	\begin{align}\label{eqn:no-full-prym-result}
	[K:\bbQ] (\dim J_1 - \dim J_{2})  - (\rank J_{1}(\OKS) - \rank J_{2}(\OKS)) \geq \left(\frac{\delta-1}{\delta}\right) \Delta\,.
	\end{align}
\end{lemma}

\begin{proof}[Proof of Lemma~\ref{lem:covers-rank}]
	Let $n = \#(G\backslash \Gamma_1(\overline{K}))$. Since $\#(G\backslash \Gamma_{2}(\overline{K})) = 1$, every point in $\# \Gamma_{2}(\overline{K})$ has exactly $\delta$ preimages in $\Gamma_{1}(\overline{K})$. Each orbit of $G\smallsetminus \Gamma_1(\overline{K})$ contains at least one preimage of each point of $\Gamma_2(\overline{K})$, so $n \leq \delta$. Similarly, for each $\fkp \in S \cup \Sigma_{\infty}$, we have $\#(G_{\fkp} \backslash \Gamma_1(\overline{K})) \leq \delta \#(G_{\fkp} \backslash \Gamma_2(\overline{K}))$. Adding $\rank \bbG_{m}(\OKS) = -1 + \sum_{\fkp \in S \cup \Sigma_{\infty}} 1$ to the formula from Lemma~\ref{lem:rank+dim-genus-0-jac} gives
	\begin{align}\label{eqn:no-full-prym-1}
	\rank J_1(\OKS) + \rank \bbG_{m}(\OKS) + n  & = \sum_{\fkp \in S \cup \Sigma_{\infty}} \#(G_{\fkp} \backslash \Gamma_1(\overline{K}))\,,  \text{ and }\\ \label{eqn:no-full-prym-2}
	\rank J_2(\OKS) + \rank \bbG_{m}(\OKS) + 1 & = \sum_{\fkp \in S \cup \Sigma_{\infty}} \#(G_{\fkp} \backslash \Gamma_2(\overline{K}))\,.
	\end{align}
	Subtracting $\delta$ times \eqref{eqn:no-full-prym-2} from \eqref{eqn:no-full-prym-1} gives
	\[
	(\delta - (\delta - 1)) \rank J_1(\OKS) - \delta  \rank J_2(\OKS) - (\delta - 1) \rank \bbG_{m}(\OKS)  - (\delta - n)  \leq 0\,.
	\]
	Since $n \geq 1$, the inequality is true after replacing $n$ with $1$. Rearranging gives
	\begin{align}\label{eqn:no-full-prym-3}
	\rank J_{1}(\OKS) - \rank J_{2}(\OKS) \leq \frac{\delta - 1}{\delta} \left(\rank J_{1}(\OKS) + \rank \bbG_{m}(\OKS) + 1\right)\,.
	\end{align}
	We also compute
	\begin{align}\label{eqn:no-full-prym-4}
	\dim J_1 - \dim J_2 = (\#\Gamma_1(\overline{K})-1) - (\#\Gamma_2(\overline{K})-1) = \frac{\delta-1}{\delta} \#\Gamma_1(\overline{K})\,.
	\end{align}
	Subtracting \eqref{eqn:no-full-prym-3} from $[K:\bbQ]$ times \eqref{eqn:no-full-prym-4}, we conclude that \eqref{eqn:no-full-prym-result} holds.
\end{proof}

After unpacking some definitions, Theorem~\ref{thm:main-theorem} is almost a corollary of Theorem~\ref{thm:rank-bounds} and Lemma~\ref{lem:covers-rank}.

\begin{proof}[Proof of Theorem~\ref{thm:main-theorem}]
	Suppose that $T$ is an $n$-BCP torus for $X_{\alpha,q}$. Choose $R_{i}', X_{i},$ and $T_{i}$ and define $\delta_{i}$ as in Lemma~\ref{lem:rank-deficiency}.
	
	If $n = 0$ so that $(X_{\alpha,q}, T) = (X_0, T_{0})$ or if $n = 1$ and $(X_{\alpha,q}, T)$ is a BC-successor of $(X_0, T_0)$ for $X_{0}/R_0'$ then in the notation of Theorem~\ref{thm:rank-bounds}, we have $R' = R_{0}'$ and $T_{0} = \Res_{R_{0}'/R_{0}} J_{X_{\alpha,q}}$. Applying Theorem \ref{thm:rank-bounds}, we have
	\begin{align*}
	\rank T_{0}(R_{0}) & \leq \left([K_0': \bbQ] - \frac{1}{2} + \varepsilon \right) (q-2)\,, \\
	\dim T_{0} & \geq [K_0':\bbQ](q-2)\,.
	\end{align*}
	Choosing any $\varepsilon < 1/2$ and $q$ sufficiently large, we may arrange that
	\[
	\sum_{i=0}^{n} \max(0, \delta_{i}) \geq \delta_{0} \geq \left(\frac{1}{2} - \varepsilon\right)(q-2) \geq [K_0':\bbQ] \geq \dim(T \cap j(\Res_{R/R_{0}} X))\,.
	\]
	Thus, $T$ is not a BCP obstruction for RoS Chabauty applied to $X_{\alpha,q}$.
	
	Otherwise, we may write $(X_{\alpha,q},T)$ as a BC-successor of some  $(X_{n-1}, T_{n-1})$ (possibly equal to $(X_{\alpha,q},T)$) which is a P-successor of some $(X_{n-2}, T_{n-2})$ (and where $X_{n-1} \to X_{n-2}$ is not an isomorphism.) For ease of notation, say $X_{n-1}$ and $X_{n-2}$ are defined over $R'$, the ring of $S$-integers in $K'$.
	
	By Theorem~\ref{thm:rank-bounds}, taking $\varepsilon <  1/4$ and $q$ sufficiently large, we may assume
	\[
	\rank J_{n-1}(R') \leq \left([K':\bbQ] - \frac{1}{2} + \varepsilon \right) (q-2)  \leq \left([K':\bbQ] - \frac{1}{4}\right) q - \rank \bbG_{m}(R') - 1\,.
	\]
	In particular, 
	\[
	[K':\bbQ] q - (\rank J_{n-1}(R') + \rank \bbG_{m}(R') + 1) \geq \frac{q}{4}\,.
	\]
	
	By construction, the punctures of $X_{n-1}$ form a single $\Gal(\overline{K'}/K')$-orbit so the same is true of $X_{n-2}$. We claim that the number of punctures of $X_{n-1}(\overline{K'})$ and $X_{n-2}(\overline{K'})$ cannot be equal. If they were, then the map $\overline{X_{n-1}} \to \overline{X_{n-2}}$ would be totally ramified at every puncture. But by the Riemann-Hurwitz formula, a map between genus zero curves is totally ramified at at most $2$ points. Applying Lemma~\ref{lem:covers-rank}, we have $\delta \geq 2$ and so
	\[
	[K':\bbQ] ( \dim J_{n-1} - \dim J_{n-2}) - (\rank J_{n-1}(R') - \rank J_{n-2}(R')) \geq \left(\frac{\delta - 1}{\delta}\right) \frac{q}{4} \geq \frac{q}{8}\,.
	\]
	Taking $q$ sufficiently large, we may arrange that
	\[
	\sum_{i=0}^{n} \max(0, \delta_{i}) \geq \delta_{n-1} \geq \frac{q}{8} \geq [K:\bbQ]  = 
	\dim\,j(\Res_{R/R_{0}} X_{\alpha,q}) \geq \dim\, T \cap j(\Res_{R/R_{0}} X)\,.
	\]
	Again, we see that $T$ is not a BCP obstruction for RoS Chabauty applied to $X_{\alpha,q}$.
\end{proof}

\section{No subgroup obstructions to RoS Chabauty for $\mathbb P^1 \smallsetminus \{x: x^q - \alpha = 0\}$}\label{sec:RoS-chab+descent-2}

Let $q$ be a prime number. For any $\alpha \in \OKS^{\times}$ which is not a $q$th power in the Galois closure of $K$, let $\fkm_{\alpha,q} = \{x: x^q - \alpha = 0\}$ and let $\Gamma_{\alpha,q}$ be the closure of the support of $\fkm_{\alpha,q}$ in $\bbP^1_{\OKS}$.

In this section, we show that when $q$ is sufficiently large, there are no subgroup obstructions to RoS Chabauty for $\mathbb P^1_\OKS \smallsetminus \Gamma_{\alpha,q}$, thereby providing evidence that the set $(\mathbb P^1_\OKS \smallsetminus \Gamma_{\alpha,q})(\cO_{K} \otimes \bbZ_{p})_{S,1}$ of $p$-adic points cut out by RoS Chabauty is finite. Our main goal is to prove Theorem~\ref{thm:main-theorem-2}.

Our strategy leverages the fact that the Galois group of the Galois closure of $K[\sqrt[q]{\alpha}]$ is non-abelian. We will use the representation theory of a copy of $\bbZ /q \bbZ \rtimes \bbZ/2\bbZ$ inside of this Galois group to control the $\rank T(R_{0})$ when $T$ is a subtorus of the restriction of scalars of the generalized Jacobian of $\mathbb P^1_R \smallsetminus \Gamma_{\alpha,q}$\,.

\begin{lemma}\label{lemma:rank-of-subtori}
	Let $K$ be a number field with Galois closure $L$ and let $q$ be an odd prime. Fix any $\alpha \in \cO_{K}$ such that $x^q - \alpha = 0$ has no solutions in $L$. Let $T/\bbZ$ be a subtorus of the N\'eron model of $\Res_{K[\sqrt[q]{\alpha}]/\mathbb Q} \mathbb G_{m} / \Res_{K/\mathbb Q} \mathbb G_{m}$. Then, $\dim T$ is a multiple of $q-1$ and $\rank T(\mathbb Z) = \frac{1}{2} \dim(T)$.
\end{lemma}

\begin{proof}
	Let $L' = L[\sqrt[q]{\alpha}, \zeta_{q}]$ be the Galois closure of $L[\sqrt[q]{\alpha}]$. Let $X(T) = \mathrm{Hom}_{L'}(T, \mathbb G_{m})$ be the character lattice of $T$ and let $\chi_{T}$ be the associated complex representation of $\Gal(L'/\mathbb Q)$. We will study $T$ by considering $\chi_{T}$ restricted to certain subgroups of $\Gal(L'/\mathbb Q)$.
	
	To begin, note that $\Gal(L'/\mathbb Q) = (\mathbb Z/q\mathbb Z) \rtimes \Gal(L[\zeta_{q}]/\mathbb Q)$. The normal subgroup $\mathbb Z/q\mathbb Z \subset \Gal(L'/\mathbb Q)$ acts on the character lattice of $\Res_{K[\sqrt[q]{\alpha}]/\mathbb Q} \mathbb G_{m}$ as $[K:\mathbb Q]$ copies of the regular representation and acts trivially on the $[K:\mathbb Q]$-dimensional subspace corresponding to the character lattice of $\Res_{K/\mathbb Q} \mathbb G_{m}$. Since we quotient by this subspace we have, $(X(T))^{\Gal(L'/\mathbb Q)} = \{0\}$. Since $T$ is defined over $\mathbb Q$, the non-trivial characters of $\mathbb Z/q \mathbb Z$ all appear to the same multiplicity in $X(T)_{\mathbb C}$, whence $\dim T$ is a multiple of $q-1$. 
	
	Let $H \subset \Gal(L'/\mathbb Q)$ be the stabilizer of any of the complex places of $L'$ under the natural action of $\Gal(L'/\mathbb Q)$ on its infinite places. The group $H$ is abstractly isomorphic to $\mathbb Z/2\mathbb Z$ and the representation of $\Gal(L'/\bbQ)$ on the space of formal linear combinations of infinite places of $L'$ is $\mathrm{Ind}^{\text{Gal}(L'/\bbQ)}_{H} \mathbf{1}$. 	Arguing as in the proof of Lemma~\ref{lem:rank-anisotropic} and using the fact that $T$ is anisotropic for the first equality and applying Frobenius reciprocity for the second equality, we have  
	\[
	\rank T(\bbZ) = \left \langle \chi_{T}, \mathrm{Ind}^{\text{Gal}(L'/\bbQ)}_{H} \mathbf{1} \right \rangle_{\text{Gal}(L'/\bbQ)} =  \left \langle \mathrm{Res}^{\text{Gal}(L'/\bbQ)}_{H} \chi_{T},  \mathbf{1} \right \rangle_{H}\,.
	\]
	So, the rank of $T(\bbZ)$ is equal to the multiplicity of the trivial representation in the restriction of $\chi_{T}$ to $H$. To study this multiplicity, we first restrict $\chi_{T}$ to a representation of $\bbZ/q\bbZ \rtimes H \subset \Gal(L'/\bbQ)$. Note that this semidirect product is not a direct product. 
	
	The irreducible representations of $\mathbb Z/q\mathbb Z \rtimes \mathbb Z/2\mathbb Z$ consist of two $1$-dimensional representations in which $\mathbb Z/q \mathbb Z$ 
	acts trivially and $(q-1)/2$ different $2$-dimensional representations. Each $2$-dimensional representation has the form $\mathrm{Ind}_{\mathbb Z/q\mathbb Z}^{\mathbb Z/q\mathbb Z \rtimes \mathbb Z/2\mathbb Z} \chi$ for some nontrivial character $\chi$ of $\mathbb Z/q \mathbb Z$ which depends on the representation. The restrictions of these $2$-dimensional characters to $H$ break up as a sum of the trivial character and the unique non-trivial character of $H$. 
	
	Now, we have already seen that the fixed subspace $(\chi_{T})^{\mathbb Z/q \mathbb Z}$ is trivial, so after restricting to $\mathbb Z/q\mathbb Z \rtimes H$, 
	the space $\chi_{T}$ decomposes as a sum of $2$-dimensional irreducible representations. By the discussion in the previous paragraph, upon further restricting $\chi_{T}$ to $H$, it decomposes into an equal number of copies of the trivial representation and the non-trivial representation. We conclude that $\rank T(\mathbb Z) = \frac{1}{2} \dim(T)$.
\end{proof}

\begin{lemma}\label{lem:rank-of-subtori-2}
	Let $K$ be a number field and let $S$ be a finite set of finite places. For any $\varepsilon > 0$, there exists $C \in \bbR$ such that for all primes $q > C$ the following holds:
	
	Fix any $\alpha \in \OKS$ such that $\alpha$ is not a $q$th power in $\OKS$. Let $T/\bbZ$ be a subtorus of the N\'eron model of $\Res_{K[\sqrt[q]{\alpha}]/\mathbb Q} \mathbb G_{m} / \Res_{K/\mathbb Q} \mathbb G_{m}$. Then, $\rank T(R_{0}) \leq \left(\frac{1}{2} + \varepsilon\right) \dim(T)$.
\end{lemma}

\begin{proof}
	Let $L$ be the Galois closure of $K$. For $q$ large enough, the $q$th powers in $L$ and $K$ are the same, so we may assume $\alpha$ is not a $q$th power in $L$.	
	
	By the same argument as in the proof of Lemma~\ref{lemma:rank-of-subtori}, the contribution of the infinite places to $\rank(T(R_{0}))$ is $\frac{1}{2} \dim T$, so it suffices to consider the contribution of the finite places.
	
	Let $S'$ be the set of primes of $K[\sqrt[q]{\alpha}]$ above $S$ and let $R' = \cO_{K[\sqrt[q]{\alpha}], S'}$. Since $T$ is a subtorus of the N\'eron model of $\Res_{K[\sqrt[q]{\alpha}]/\mathbb Q} \mathbb G_{m} / \Res_{K/\mathbb Q} \mathbb G_{m}$, this contribution is bounded by the contribution of the finite places to the rank of $\cO_{K[\sqrt[q]{\alpha}],S'}^{\times}/\cO_{K,S}^{\times}$. So, we have
	\begin{align*}
	\rank  T(R_{0}) - \rank T(\mathbb Z) & \leq \frac{1}{2} \dim(T) + \rank \cO_{K[\sqrt[q]{\alpha}],S'}^{\times}  - \rank \cO_{K[\sqrt[q]{\alpha}]}^{\times} - \rank \cO_{K,S}^{\times} + \rank \cO_{K}^{\times} \\
	& = \frac{1}{2} \dim(T) + \#S' - \# S\,.
	\end{align*}
	
	Now, for any prime $\fkp \in S$. Let $\kappa_{\fkp}$ be the residue field of $K_{\fkp}$. Let $a_{\fkp} = [\kappa_{\fkp}(\zeta_q):\kappa_{\fkp}]$. 
	Then, $a_{\fkp}$ is the order of $\# \kappa_{\fkp}$ in $(\bbZ/q\bbZ)^\times$, so 
	\[
	a_{\fkp} \geq \left \lceil \frac{\ln (q - 1)}{\ln(\#\kappa_{\fkp})} \right \rceil\,.
	\]
	For a finite place $\fkp$ of $K$ lying above $q$, there is a single prime of $K(\sqrt[q]{\alpha})$ above $\fkp$. 
	For a finite place $\fkp$ of $K$ not lying over $q$, 
	\[
	\#\{\text{primes } \fkP \text{ of } K(\sqrt[q]{\alpha}) \text{ above } \fkp\} = \begin{cases}
	1, & \text{ if } \overline{\alpha} \notin \kappa_{\fkp}^{\times q}\,, \\
	1 + (q-1)/a_{\fkp}, & \text{ otherwise }\,.
	\end{cases}
	\]
	In particular, as $q \to \infty$, the total number of primes of $K(\sqrt[q]{\alpha})$ is $O(q/\log q)$, where the implicit constant depends only on the residue fields of the primes in $S$.
\end{proof}

\begin{proof}[Proof of Theorem~\ref{thm:main-theorem-2}]
	Let $\varepsilon = 1/4$ and choose $q$ greater than both $4[K:\bbQ]$ and the resulting constant $C$ in Lemma~\ref{lem:rank-of-subtori-2}.	The generalized Jacobian $J_{\alpha,q}$ of $X_{\alpha,q}$ is the N\'eron model of $\Res_{K[\sqrt[q]{\alpha}]/\mathbb Q} \mathbb G_{m} / \Res_{K/\mathbb Q} \mathbb G_{m}$. 
	
	Applying Lemma~\ref{lem:rank-of-subtori-2}, we see that for any subtorus $T \subset J_{\alpha,q}$ we have $\rank T(R_{0}) < \frac{3}{4} \dim T$. In particular, $\dim\overline{T(R_{0})} < \frac{3}{4} \dim T$. Since $\dim T$ is a multiple of $q-1$ by Lemma~\ref{lemma:rank-of-subtori}, we have
	\[
	\dim T - \overline{T(R_{0})} > \frac{1}{4} \dim T \geq \frac{1}{4} (q-1)\,. 
	\]
	On the other hand, for any translate of $T$, we have
	\[
	\dim ((x\cdot T) \cap j(\Res_{R/R_{0}} X_{\alpha,q})) \leq \dim j(\Res_{R/R_{0}} X_{\alpha,q})) = [K:\bbQ]\,. 
	\]
	Since  $q \geq 4 [K:\bbQ] + 1$, the result follows.
\end{proof}

\bibliographystyle{amsalpha}

\bibliography{biblio}

\end{document}